\documentclass[11pt, notitlepage]{amsart}
\usepackage{a4wide}
\usepackage{amsthm,amsmath,amssymb,amscd,mathrsfs}

\newtheorem{thm}{Theorem}[section]
\newtheorem{prop}[thm]{Proposition}
\newtheorem{lemma}[thm]{Lemma}
\newtheorem{cor}[thm]{Corrolary}
\theoremstyle{definition}
\newtheorem*{defin}{Definition}
\newtheorem*{ack}{Acknowledgments}
\newtheorem*{convention}{Convention}
\newtheorem{exam}[thm]{Example}
\theoremstyle{remark}
\newtheorem{remark}[thm]{Remark}
\def\pntje{Then the following are equivalent:}

\def\ifff{if and only if }
\def\letsM{Let us fix a $(*)$-elementary submodel $M$}
\def\letsMX{Let us fix a $(*)$-elementary submodel $M$ containing $X$}
\def\letsMXA{Let us fix a $(*)$-elementary submodel $M$ containing $X$ and $A$. }
\def\pp{For a suitable elementary submodel $M$ the following holds:\\}
\def\ppX{Then whenever $M$ contains $X$, it is true that }
\def\ppXA{Then whenever $M$ contains $X$ and a set $A\subset X$, it is true that}
\def\ppXf{Then whenever $M$ contains $X$ and $f$, it is true that for every $x\in X_M\cap G$:}
\newcommand{\ppXfD}[2][G]{Then whenever $M$ contains $X$, $f$ and #2, it is true that for every $x\in X_M\cap #1$:}
\newcommand{\ppXfJ}[1]{Then whenever $M$ contains $X$, $f$ and #1, it is true that }
\newcommand{\ppXD}[2][]{Then whenever $M$ contains $X$, #1 and #2, it is true that}
\newcommand{\ppXJ}[1][]{Then whenever $M$ contains $X$, #1, it is true that}

\def \conv {\operatorname{conv}}
\def \cl {\operatorname{cl}}
\def \Int {\operatorname{Int}}
\def \sspan {\operatorname{span}}
\def \suppt {\operatorname{suppt}}
\def \rng {\operatorname{Rng}}
\def \dom {\operatorname{Dom}}

\def \d {\operatorname{d}}

\def\M{\mathcal{M}}
\def\C{\mathcal{C}}
\def\R{\mathcal{R}}
\def\L{\mathscr{L}}

\def\en{\mathbb N}
\def\zet{\mathbb Z}
\def\er{\mathbb R}
\def\qe{\mathbb Q}

\def\iff{\leftrightarrow}
\def\eps{\varepsilon}  
\def\ov{\overline}

\begin{document}
\author{Marek C\'uth}
\title{Separable reduction theorems by the method of elementary submodels}
\thanks{The work was supported by the grant SVV-2010-261316}
\email{cuthm5am@karlin.mff.cuni.cz}
\address{Charles University, Faculty of Mathematics and Physics, Sokolovsk\'a 83, 186 75 Praha 8 Karl\'{\i}n, Czech Republic}
\begin{abstract}We introduce an interesting method of proving separable reduction theorems - the method of elementary submodels. We are studying whether it is true that a set (function) has given property if and only if it has this property with respect to a special separable subspace, dependent only on the given set (function). We are interested in properties of sets ``to be dense, nowhere dense, meager, residual or porous'' and in properties of functions ``to be continuous, semicontinuous or Fr\'echet differentiable''. Our method of creating separable subspaces enables us to combine our results, so we easily get separable reductions of function properties such as ``be continuous on a dense subset'', ``be Fr\'echet differentiable on a residual subset'', etc. Finally, we show some applications of presented separable reduction theorems and demonstrate that some results of Zaj\'{\i}\v{c}ek, Lindenstrauss and Preiss hold in nonseparable setting as well.\end{abstract}
\maketitle
\section{Introduction}\
The method of elementary submodels is a set-theoretical method which can be used in various branches of mathematics. A.Dow in \cite{dow} illustrated the use of this method in topology, W.Kubi\'s in \cite{kubis} used it in functional analysis, namely to construct projections on Banach spaces. In the present work we slightly simplify and precise the method of elementary submodels from \cite{kubis} and we study whether this method can be used to prove separable reduction theorems which had not been proven by other (more standard) methods.

As an success in this way may be considered the following three results. First, we show that porosity is a separable determined property. Second, we extend the validity of Zaj\'{\i}\v{c}ek's result [\ref{zajicDif2}; Proposition 3.3] from spaces with separable dual to the general Asplund spaces. And finally, we extend the validity of Preiss's and Lindenstrauss's result [\ref{preis}; Theorem 4.8] from spaces $c_0$ and $\C(K)$ with a countable compact $K$ to the spaces $c_0(\Gamma)$ and $\C(K)$ with a general scattered compact $K$.

It seems that the main advantages of the concept of elementary submodels are:
\begin{itemize}
 \item finite number of results may be combined
 \item the results may be used for more than one space at the same time (having two spaces $X$ and $Y$ which are dependent on each other in some way, we are able to use results for the spaces $X$, $Y$ and combine them together).
\end{itemize}
Thus, the real strength of this method is revealed when we have proven enough results to combine them together.
 
The structure of the work is as follows: first we introduce elementary submodels and show some general results about them. Then we point out how this method is connected with the question of separable subspaces. Next, we collect properties of sets and functions which are separably determined. In the end we produce two extensions of the results contained in \cite{zajicDif2} and \cite{preis} using the method of elementary submodels.

Below we recall most relevant notions, definitions and notations.

We denote by $\omega$ the set of all natural numbers (including $0$), by $\en$ the set $\omega\setminus\{0\}$, by $\er_+$ the interval $(0,\infty)$ and $\qe_+$ stands for $\er_+\cap\qe$. Whenever we say that a set is countable, we mean by this that the set is either finite or infinite and countable. If $f$ is a mapping then we denote by $\rng f$ the range of $f$ and by $\dom f$ the domain of $f$. By writing $f:X\to Y$ we mean that $f$ is a mapping with $\dom f = X$ and $\rng f \subset Y$. By the symbol $f\upharpoonright_{Z}$ we mark the restriction of the mapping $f$ to the set $Z$. The closure (resp. interior) of a set $A$ we denote by $\ov{A}$ (resp. $\Int{(A)}$); the interior relative to a subspace $Y$ we denote by $\Int_Y{(A)}$. 

If $\langle X,\rho\rangle$ is a metric space, we denote by $U(x,r)$ the open ball, i.e. the set $\{y\in X: \rho(x,y) < r\}$. We shall consider normed linear spaces over the field of real numbers (but many results hold for complex spaces as well). If $X$ is a normed linear space and $A\subset X$, we mean by $\conv{A}$ the convex hull of $A$, by $\cl_w(A)$ the weak closure of $A$ and by $\sspan{A}$ the linear span of $A$. $S_X$ is the unit sphere in $X$, i.e. the set $\{x\in X:\; \|x\| = 1\}$. $X^*$ stands for the dual space of $X$. By $\C(K)$ we mean the space of continuous functions on the compact space $K$.
\section{Elementary submodels}\
In this section we introduce the method of creating sets with some special properties using elementary submodels. First we define what those elementary submodels are. Then we show which properties they can have. The method discussed in this article is based on a set-theoretical theorem \ref{tCountModel}. It is a combination of the Reflection Theorem and  L\"owenheim--Skolem Theorem. We refer reader to Kunen's book \cite{kunen}, where further details can be found.

The idea to use this method in functional analysis comes from the Kubi\'s's article \cite{kubis}. Some of the following results are therefore based on this article and slightly modified to our situation (namely lemma \ref{lUniqueInM} and propositions \ref{pCountableSubset}, \ref{pIsSubspace}, \ref{lLPM} and \ref{lCKM}).

Let us first recall some definitions:

Let $N$ be a fixed set and $\phi$ formula. Then {\em relativization of $\phi$ to $N$} is a formula $\phi^N$ which is a formula obtained from $\phi$ by replacing each quantifier of the form ``$\forall x$'' by ``$\forall x\in N$'' and each quantifier of the form ``$\exists x$'' by ``$\exists x\in N$''.

As an example, if
$$\phi := \forall x\; \forall y\; \exists z\; ((x\in z) \;\wedge\; (y\in z)) $$
and $N=\{a,b\}$, then the relativization of the formula $\phi$ to $N$ is
$$\phi^N = \forall x\in N\; \forall y\in N\; \exists z\in N\; ((x\in z)\; \wedge\; (y\in z))$$
It is clear that $\phi$ is satisfied, but $\phi^N$ is not.

If $\phi(x_1,\ldots,x_n)$ is a formula with all free variables shown, then {\em $\phi$ is absolute for $N$} if and only if
$$\forall a_1,\ldots,a_n\in N\quad (\phi^N(a_1,\ldots,a_n) \leftrightarrow \phi(a_1,\ldots,a_n))$$

A list of formulas, $\phi_1,\ldots,\phi_n$, is said to be {\em subformula closed} if and only if every subformula of a formula in the list is also contained in the list.

Any formula in the set theory can be written using symbols $\in,=,\wedge,\vee,\neg,\rightarrow,\leftrightarrow,\exists,(,),[,]$ and symbols for variables. Let us assume a subformula closed list of formulas $\phi_1,\ldots,\phi_n$ is written in this way. Then it is not difficult to show, that the absoluteness of $\phi_1,\ldots,\phi_n$ for $N$ in other words says, that those formulas don't create any new sets in $N$. This result is contained in the following lemma (a proof can be found in [\ref{kunen}, Chapter IV Lemma 7.3]):

\begin{lemma}\label{lKunen}
 Let $N$ be a set and $\phi_1,\ldots,\phi_n$ subformula closed list of formulas (formulas containing only symbols $\in,=,\wedge,\vee,\neg,\rightarrow,\leftrightarrow,\exists,(,),[,]$ and symbols for variables). \pntje
  \begin{itemize}
    \item[(i)] $\phi_1,\ldots,\phi_n$ are absolute for $N$
    \item[(ii)] Whenever $\phi_i$ is of the form $\exists x\phi_j(x,y_1,\ldots y_l)$ (with all free variables shown), then
  \end{itemize}
  $$\quad\forall y_1,\ldots y_l\in N\left[\exists x\ (\phi_j(x,y_1,\ldots y_l))\rightarrow(\exists x\in N)(\phi_j(x,y_1,\ldots y_l))\right]$$
\end{lemma}

The most important result from the set theory for us will be the following theorem (a proof can be found in [\ref{kunen}, Chapter IV Theorem 7.8]).

\begin{thm}\label{tCountModel}
 Let $\phi_1,\ldots,\phi_n$ be any formulas and $X$ any set. Then there exists a set $M\supset X$ such, that
 $$(\phi_1,\ldots,\phi_n \text{ are absolute for }M)\quad \wedge\quad (|M|\leq \max(\omega,|X|))$$ 
\end{thm}

The set from previous theorem will be often used throughout the paper. Therefore we will use the following definition:

\begin{defin}
 Let $\phi_1,\ldots,\phi_n$ be any formulas and let $X$ be any countable set. Let $M\supset X$ be a countable set satisfying that $\phi_1,\ldots,\phi_n$ are absolute for $M$. Then we say that {\em $M$ is an elementary submodel for $\phi_1,\ldots,\phi_n$ containing $X$}. We denote this by $M\prec(\phi_1,...,\phi_n;\; X)$.\\
 The relation between $X$, $\phi_1,\ldots,\phi_n$ and $M$ is often called {\em the elementarity of $M$}.
\end{defin}

Using lemma \ref{lKunen} it is easy to see that the countable union of a monotonne sequence of elementary submodels is also an elementary submodel.

\begin{lemma}\label{lCupM}
 Let $\varphi_1,\ldots,\varphi_n$ be a subformula closed list of formulas and let $X$ be any countable set. Let $\{M_k\}_{k\in\omega}$ be a sequence of sets satisfying
 \begin{itemize}
	\item[(i)] $M_i\subset M_j,\quad i\leq j,$
	\item[(ii)] $\forall k\in\omega:\; M_k\prec(\varphi_1,...,\varphi_n;\; X).$	
 \end{itemize}
 Then for $M:=\bigcup_{k\in\omega}{M_k}$ it is true, that also $M \prec(\varphi_1,...,\varphi_n;\; X)$.
\end{lemma}
\begin{proof}
 It is an easy consequence of lemma \ref{lKunen}.
\end{proof}

Let $\phi(x_1,\ldots,x_n)$ be a formula with all free variables shown and let $M$ be some elementary submodel for $\phi$. Supposing we want to use the absoluteness of $\phi$ for $M$ efficiently, we need to know that a lot of sets are elements of $M$. The reason is that having $a_1,\ldots,a_n\in M$, the validity of $\phi(a_1,\ldots,a_n)$ and $\phi^M(a_1,\ldots,a_n)$ coincides. Therefore, when working with elementary submodels, it is our first aim to force elementary submodel to contain as many objects as possible. Let's see a simple example, how this can be achieved.

\begin{exam}
 Let us have the following formulas:
    $$\varphi_1(x,a) := \forall z (z\in x \iff ((z\in a)\vee(z=a)))$$
    $$\varphi_2(a) := \exists x \varphi_1(x,a)$$
 Then for all sets $M$ satisfying $M\prec(\varphi_1,\varphi_2;\; \emptyset)$ it is true that whenever we have $a\in M$, then $a\cup\{a\}\in M$.
\end{exam}
\begin{proof}
 Fix $a\in M$. Then $\varphi_2(a)$ is satisfied (the set $x$ satisfying $\varphi_1(x,a)$ is $a\cup\{a\}$). From the absoluteness of $\varphi_2$ for $M$ we get, that there exists $x\in M$ satisfying $\varphi^M_1(x,a)$. Let us fix one such $x\in M$. It is true that $\varphi_1^M(x,a)$, and therefore (using absoluteness of $\varphi_1$) $\varphi_1(x,a)$ is satisfied as well. But the only possibility how $\varphi_1(x,a)$ can be satisfied is that $x = a\cup\{a\}$. Therefore $a\cup\{a\}\in M$.
\end{proof}

The preceeding example can be generalized to the following lemma:

\begin{lemma}\label{lUniqueInM}
 Let $\phi(y,x_1,\ldots,x_n)$ be a formula with all free variables shown and let $X$ be a countable set. Let $M$ be a fixed set, $M\prec(\phi, \exists y\phi(y,x_1,\ldots,x_n);\; X)$ and let $a_1,\ldots,a_n \in M$ be such that there exists only one set $u$ satisfying $\phi(u,a_1,\ldots,a_n)$. Then $u\in M$.
\end{lemma}
\begin{proof}
 Using the absoluteness of $\exists y\phi(y,x_1,\ldots,x_n)$ there exists $y\in M$ satisfying $\phi^M(y,a_1,\ldots,a_n)$. Using the absoluteness of $\phi$ we get, that for this $y\in M$ the formula $\phi(y,a_1,\ldots,a_n)$ holds. But such $y$ is unique and therefore $u=y\in M$.
\end{proof}

Using this lemma we can force the elementary submodel $M$ to contain all the needed objects created (uniquely) from elements of $M$. As an example, let us see how it is possible to force $M$ to contain its finite subsets and natural numbers.

\begin{prop}\label{pNaturalNumbersInM}
 Let us have the following formulas:
	\begin{align*}
		\varphi_1 		& := \forall z (z\in x \iff z\neq z)\\
		\varphi_{1E} 	& := \exists x \varphi_1(x)\\
		\varphi_2			& := \forall z (z\in x \iff ((z\in u)\vee(z=v)))\\
		\varphi_{2E}	& := \exists x \varphi_2(x,u,v)		
	\end{align*}
 Then for any nonempty countable set $X$ holds:
  \begin{itemize}
	\item[(i)] If $M\prec(\varphi_1, \varphi_{1E};\; X)$, then $\emptyset \in M$.
	\item[(ii)] If $M\prec(\varphi_2, \varphi_{2E};\; X)$, then for every $u,v\in M$ is $u\cup\ \{v\}\in M$.
	\item[(iii)] If $M\prec(\varphi_1, \varphi_{1E},\varphi_2, \varphi_{2E};\; X)$, then $\omega \subset M$.
	\item[(iv)] If $M\prec(\varphi_1, \varphi_{1E},\varphi_2, \varphi_{2E};\; X)$, then for every finite set $s\subset M$ is $s\in M$.
  \end{itemize}
\end{prop}
\begin{proof}
 $(i)$ and $(ii)$ follow immediately from the lemma \ref{lUniqueInM}; $(iii)$ follows from $(i)$ and $(ii)$ by induction on $n$; $(iv)$ follows from $(i)$ and $(ii)$ by induction on the cardinality of $s$.
\end{proof}

It would be very laborious and pointless to use only the basic language of the set theory. For example, we often write $x < y$ and we know, that in fact this is a shortcut for a formula $\varphi(x,y,<)$ with all free variables shown. In the following text we will use this extended language of the set theory as we are used to.

We will use the following convention.
\begin{convention}\ \label{conventionM}
 Whenever we say\\[8pt]
 {\em for a suitable elementary submodel $M$ (the following holds...)},\\[8 pt]
 we mean by this\\[8pt]
 {\em there exists a list of formulas $\phi_1,\ldots,\phi_n$ and a countable set $Y$ such that for every $M\prec(\phi_1,\ldots,\phi_n;\;Y)$ (the following holds...)}.\\
\end{convention}

When using this new terminology, we lose the information about formulas $\phi_1,\ldots,\phi_n$ and the set $Y$. Anyway, this is not important in applications.

\begin{remark}\label{rCombine}
 Let us have finite number of sentences $T_1(a),\ldots,T_n(a)$. Let us assume that whenever we fix $i\in \{1,\ldots,n\}$, then for a suitable elementary submodel $M_i$ the sentence $T_i(M_i)$ is satisfied. Then it is easy to verify, that for a suitable model $M$ the sentence $$T_1(M) \text{ and } \ldots \text{ and } T_n(M)$$ is satisfied (it is enough to put together all the lists of formulas and all the sets from the definition above).\\
 In other words, we are able to combine any finite number of results we have proven using the technic of elementary submodels.  
\end{remark}

Let us see some general results about suitable elementary submodels.

\begin{prop}\label{pRngInM}\pp
 Let $f$ be a function such that $f\in M$. Then
  \begin{itemize}
    \item[(i)] $\dom{f} \in M$
    \item[(ii)] $\rng{f}\in M$
    \item[(iii)] $(\forall x\in M\cap\dom{f})\;(f(x)\in M)$
  \end{itemize}
\end{prop}
\begin{proof}
 Let us fix an elementary submodel $M$ for formulas marked with $(*)$ in the proof below and all their subformulas. Let $f\in M$ be a function. Then  $\dom{f}$ is an object uniquely defined by the following formula (this formula is the same for all functions $f$, $f$ is a free variable in this formula)
 $$(*)\qquad (\exists D)(\forall x)(x\in D \iff (\exists y:\; f(x)=y)),$$
 and so by the lemma \ref{lUniqueInM},  $\dom{f}\in M$. Similarly, $\rng{f}$ is object uniquely defined by the formula
 $$(*)\qquad (\exists R)(\forall y)(y\in R \iff (\exists x:\; f(x)=y)).$$
 By the absoluteness of the formula
  $$(*)\qquad(\forall x\in D)\; (\exists y: f(x) = y)$$
 we get that $(iii)$ holds.
\end{proof}

In the sequel we will often start our proofs in the same way. Therefore, by saying ``Let us fix a $(*)$-elementary submodel $M$ [containing $A_1,\ldots,A_n$]'' we will understand the following:\\[8 pt]
``Let us have formulas $\varphi_1, \varphi_{1E}, \varphi_2, \varphi_{2E}$ from the proposition \ref{pNaturalNumbersInM} and all the formulas  marked with $(*)$ in all the preceeding proofs (and all their subformulas). Add to them formulas marked with $(*)$ in the proof below (and all their subformulas). Denote such a list of formulas by $\phi_1,\ldots,\phi_n$. Let us fix a countable set $X$ containing the sets $\omega$, $\zet$, $\qe$, $\qe_+$, $\er$, $\er_+$ and all the common operations and relations on real numbers ($+$, $-$, $\cdot$, $:$, $<$). Fix an elementary submodel $M$ for formulas $\phi_1,\ldots,\phi_n$ containing $X$ [such that $A_1,\ldots,A_n\in M$]''.\\[8 pt]
Thus, having such a ``$(*)$-elementary submodel'', we are allowed to use the results of all the preceeding theorems and propositions.

Using this new agreement, let us prove another general proposition.

\begin{prop}\label{pCountableSubset}\pp
  \begin{itemize}
	\item[(i)] Let $S$ be a finite set. Then $$S\in M \iff S\subset M.$$
	\item[(ii)] Let $S$ be a countable set. Then $$S\in M\rightarrow S\subset M.$$
	\item[(iii)] For every natural number $n>0$ and for arbitrary $(n+1)$ sets  $a_0,\ldots,a_n$ it is true,that
				$$a_0,\ldots,a_n\in M \iff \langle a_0,\ldots,a_n\rangle\in M.$$
	\item[(iv)] If $A,B\in M$, then $A\cap B\in M$, $B\setminus A\in M$ and $A\cup B\in M$.
  \end{itemize}
\end{prop}
\begin{proof}
 \letsM. Let us prove that $(ii)$ holds. Let $S\in M$ be a countable set. If $S=\emptyset$, then $S\subset M$. If $S\neq\emptyset$, then
 $$(*)\qquad (\exists f)\,(f\text{ is a function from }\omega\text{ onto }S).$$
 Thus, from the elementarity of $M$, there exists $f\in M$ satisfying
 $$(f\text{ is a function from }\omega\text{ onto }S)^M$$
 Fix one such function $f$. Then, using the elementarity of $M$ again, we get that $f$ is a function from $\omega$ onto $S$. Because $f$ is a function with $\rng f = S$ and $\dom f = \omega \subset M$, using the proposition \ref{pRngInM} it is true that $S\subset M$.

 Let us prove that $(i)$ holds. If $S\in M$ is finite, then $S\subset M$ by $(ii)$. If $S\subset M$ is finite, then $S\in M$ according to the proposition \ref{pNaturalNumbersInM}.

 $(iii)$ holds easily from $(i)$ by induction on $n\in\omega, n\geq 1$. It is enough to realize, that $\langle a_0, a_1\rangle = \{a_0,\{a_0,a_1\}\}$ and $\langle a_0,\ldots,a_n\rangle = \langle \langle a_0,\ldots,a_{n-1}\rangle, a_n\rangle$.
 
 Let us have sets $A,B\in M$. Then, using the lemma \ref{lUniqueInM} and the absoluteness of formulas
  \begin{align*}
  (*)\qquad & (\exists C)(\forall x)(x\in C\leftrightarrow x\in A \wedge x\in B),\\
  (*)\qquad & (\exists D)(\forall x)(x\in D\leftrightarrow x\in B \wedge x\notin A),\\
  (*)\qquad & (\exists E)(\forall x)(x\in E\leftrightarrow x\in A \vee x\in B),
  \end{align*}
 $(iv)$ holds.
\end{proof}
\section{Elementary submodel in the context of normed linear spaces}\label{richFamilies}
Now we are prepared for some more concrete results concerning mostly metric spaces or normed linear spaces (NLS for short). Before we proceed, let us propose the following agreements.

If $\langle X,\rho\rangle$ is a metric space (resp. $\langle X, +, \cdot, \|\cdot\|\rangle$ is a NLS) and $M$ an elementary submodel, then by saying {\em $M$ contains $X$} (or by writing $X\in M$) we mean that $\langle X,\rho\rangle\in M$ (resp. $\langle X, +, \cdot, \|\cdot\|\rangle\in M$). If $A$ is a set, then by saying that an elementary model $M$ contains $A$ we mean that $A\in M$.

If $X$ is a topological space and $M$ an elementary submodel, then we denote by $X_M$ the set $\ov{X\cap M}$.

\begin{prop}\label{pUxrInM}\pp
 Let $\langle X,\rho\rangle$ be a metric space. \ppX $U(x,r)\in M$ whenever $x\in X\cap M$ and $r\in\er_+\cap M$. 
\end{prop}
\begin{proof}
 \letsMX. Let us have $x\in X\cap M$ and $r\in\er_+\cap M$. Then $U(x,r)$ is an object uniquely determined by the following formula
 $$(*)\qquad(\exists U)(\forall z)(z\in U\; \leftrightarrow\; z\in X \wedge \rho(x,z) < r).$$
 Thus, according to the lemma \ref{lUniqueInM}, $U(x,r)\in M$.
\end{proof}

The idea of the following proposition comes from \cite{kubis}.
\begin{prop}\label{pIsSubspace}\pp
 Let $X$ be a NLS. \ppXA:
 \begin{itemize}
	\item[(i)] $\ov{\sspan(A)\cap M}$ is closed separable linear subspace of $X$.
	\item[(ii)] $\ov{\conv(A)\cap M}$ is convex set.
	\item[(iii)] If $A$ is convex, then $\ov{(A\cap M)} = \cl_w{(A\cap M)}$.
 \end{itemize}
 In particular, $X_M$ is separable subspace of $X$ and $X_M = \cl_w(X\cap M)$.
\end{prop}
\begin{proof}
 \letsMXA Then according to the proposition \ref{pCountableSubset}, $\qe\subset M$ and $\langle\er, +, -, \cdot, :, <\rangle\in M$.

 The elementary submodel $M$ contains functions $+:X\times X\rightarrow X$ and $\cdot:\er\times X\rightarrow X$. Consequently (by the proposition \ref{pRngInM}), $X\cap M$ is a $\qe$-linear subspace of $X$. Therefore $(i)$ and $(ii)$ holds. $(iii)$ follows easily from $(ii)$.
\end{proof}

Given a Banach space $X$, list of formulas $\phi_1,\ldots,\phi_n$ and a countable set $Y$, we are able to get a family of sets
$$\M(X):=\{X_M;\;M\prec(\phi_1,...,\phi_n;\; Y)\}.$$
By choosing suitable formulas $\phi_1,\ldots,\phi_n$ and suitable set $Y$, it is possible to force $\M(X)$ to be a family of closed separable subspaces of $X$ having some specific properties. One can easily join finite number of arguments (lists of formulas) and get another family of separable subspaces having the same properties as the original family and perhaps even some more.

In \cite{tiser} similar families of closed separable subspaces are used for getting separable reduction theorems. Those families are called rich. This concept has been originally introduced in \cite{borwein} by Borwein and Moors. It is possible to find further use of this method for example in \cite{moors}, where even more references may be found.
\begin{defin}
 Let $X$ be a Banach space. A family $\R$ of separable subspaces of $X$ is called {\em rich} if
 \begin{itemize}
  \item[(i)] for every increasing sequence $R_i$ in $\R$, $\ov{\bigcup_{i\in\omega} R_i}$ belongs to $\R$, and
  \item[(ii)] each separable subspace of $X$ is contained in an element of $\R$.
 \end{itemize}
\end{defin}

The connection between the notion of rich families and elementary submodels is described in the following lemma.
\begin{lemma}\label{lRichM}
 Let $X$ be a Banach space. Then there exists a list of formulas $\phi_1,\ldots,\phi_n$ and a countable set $Y$ such that for every countable set $Z$ and every list of formulas $\varphi_1,\ldots,\varphi_k$ such that $\phi_1,\ldots,\phi_n,\varphi_1,\ldots,\varphi_k$ is subformulas closed it is true that the family
 $$\M:=\{M;\;M\prec(\phi_1,...,\phi_n,\varphi_1,\ldots,\varphi_k;\; Y\cup Z)\}$$
 satisfies the following conditions:
 \begin{itemize}
  \item[(i)] the set $\{X_M;\;M\in\M\}$ is a family of closed separable subspaces of $X$,
  \item[(ii)] For every increasing sequence of elementary submodels $\{M_i\}_{i\in\omega}\subset\M$, $$\bigcup_{i\in\omega} M_i\in\M\quad\text{and}\quad\ov{\bigcup_{i\in\omega} X_{M_i}}= X_{\bigcup_{i\in\omega} M_i}.$$
  \item[(iii)] For every $V$ separable subspace of $X$ there exists $M\in\M$ such that $V\subset X_M$.
 \end{itemize}
\end{lemma}
\begin{proof}
 The existence of $\phi_1,\ldots,\phi_n$ and $Y$ such that $\{X_M;\;M\in\M\}$ is a family of closed separable subspaces follows from the proposition \ref{pIsSubspace} above. For $(ii)$, let us fix an increasing sequence $M_i$ of elementary submodels from the assumption. Then (by the lemma \ref{lCupM}) it is enough to show that $\ov{\bigcup_{i\in\omega} X_{M_i}} = X_{\bigcup_{i\in\omega} M_i}$. One inclusion follows from the fact that $\bigcup_{i\in\omega} X_{M_i} \subset \ov{\bigcup_{i\in\omega} X\cap M_i} = X_{\bigcup_{i\in\omega} M_i}$. The second one holds, because $\bigcup_{i\in\omega} X\cap M_i \subset \bigcup_{i\in\omega} \ov{X\cap M_i} = \bigcup_{i\in\omega} X_{M_i}$. Thus, $X_{\bigcup_{i\in\omega} M_i} = \ov{\bigcup_{i\in\omega} X\cap M_i} \subset \ov{\bigcup_{i\in\omega} X_{M_i}}$.
 For $(iii)$, let us take any $V$ separable subspace of $X$ and $D\subset V$ countable dense set in $V$. Then taking $M\prec(\phi_1,...,\phi_n,\varphi_1,\ldots,\varphi_k;\; Y\cup Z\cup D)$, it is true that $V\subset X_M$.
\end{proof}

In \cite{kubis} there is introduced a slightly different method of getting the elementary submodels $M$. It was proven there, that in the case of some classical Banach spaces (namely $\ell_p(\Gamma)$ and $\C(K)$) it is possible to describe the subspace $X_M$. Slightly modifying the ideas from \cite{kubis}, we get the same results in our case as well.

\begin{defin}
 Let $\Gamma$ be a set. Then we denote by $\suppt_\Gamma$ the mapping $\suppt_\Gamma:\er^\Gamma\to\Gamma$ which maps $x\in\er^\Gamma$ to $\suppt_\Gamma(x) = \{\alpha\in\Gamma;\;x(\alpha)\neq 0\}$.
\end{defin}

\begin{prop}\label{lLPM}\pp
 Let $X = \ell_p(\Gamma)$, where $1\leq p<\infty$ and $\Gamma$ is an arbitrary set. \ppXD[$\suppt_\Gamma$]{$\Gamma$}
 $$X_M = \{x\in X;\;\suppt_\Gamma(x)\subset M\}.$$
 Consequently, $X_M$ can be identified with the space $\ell_p(\Gamma\cap M)$.
\end{prop}
\begin{proof} \letsMX, $\suppt_\Gamma$, $\Gamma$. Let us mark by $A$ the set on the right-hand side.
For every $x\in X\cap M$ the set $\suppt_\Gamma(x)$ is countable and so, according to the propositions \ref{pRngInM} and \ref{pCountableSubset}, $\suppt_\Gamma(x)\subset M$. Thus, $x\in A$. So, $X\cap M\subset A$; hence $X_M\subset A$.
On the other hand, if $x\in A$ then arbitrarily close to $x$ we can find $y\in A$ such that $s = \suppt_\Gamma(y)\subset M$ is finite and $y(\alpha)\in\qe$ for $\alpha\in s$. Thus, using the proposition \ref{pCountableSubset}, $s\in M$ and $y\upharpoonright_s\in M$ (because $y\upharpoonright_s = \bigcup_{\alpha\in s}
\{\langle\alpha,y(\alpha)\rangle\}$). Using the absoluteness of the formula
$$(*)\qquad\exists z\in X(z\upharpoonright_s = y\upharpoonright_s\;\wedge\;z\upharpoonright_{\Gamma\setminus s} = 0)$$
we get, that $y\in M$. Hence $x\in \ov{X\cap M} = X_M$.
\end{proof}

Given a compact space $K$ and an arbitrary elementary submodel $M$ we define the following equivalence relation $\sim_M$ on $K$:
$$x\sim_M y \quad\iff\quad (\forall f \in \C(K)\cap M):\; f(x) = f(y).$$
We shall write $K/_M$ instead of $K/_{\sim_M}$ and we shall denote by $q^M$ the canonical quotient map. It is not hard to check that $K/_M$ is a compact Hausdorff space (see \cite{kubis}).

Observe that we can identify the spaces $\{\varphi\circ q^M:\;\varphi\in \C(K/_M)\}$ and $\C(K/_M)$. Indeed, when we define the mapping
$$F(\varphi):=\varphi\circ q^M,\quad\varphi\in\C(K/_M),$$
then it is obvious that $F$ is an isometric mapping from $\C(K/_M)$ onto $\{\varphi\circ q^M:\;\varphi\in \C(K/_M)\}$.

\begin{lemma}\label{lCKM}\pp
 Let $K$ be a compact space and $X = \C(K)$. Let us denote by $\cdot$ the operation of pointwise product of functions in $\C(K)$. \ppXD[$\cdot$]{$K$}
  $$X_M = \{\varphi\circ q^M:\;\varphi\in \C(K/_M)\}.$$
Consequently, we can identify $X_M$ with the space $\C(K/_M)$, where $K/_M$ is a metrizable compact space.
\end{lemma}
\begin{proof}
 \letsMX, $\cdot$ , $K$. Let us mark by $Y$ the set on the right-hand side. For a given function $f\in\C(K)\cap M$ we define
$$\varphi([x]_M):=f(x),\quad x\in K.$$
It is easy to verify that $\varphi$ is a continuous function. Consequently, $f\in Y$ and $X_M\subset Y$.

For the proof of the second inclusion, let us identify $X_M$ with a subspace of $\C(K/_M)$. Then, according to the propositions \ref{pIsSubspace} and \ref{pRngInM}, $X_M$ is a closed subspace closed under the operation $\cdot$. From the definition of $\sim_M$ it follows that $X_M$ separates points in $K/_M$. Using the aboluteness of the formula
$$(*)\qquad\forall c\in\er\; \exists f\in X (\forall x\in K: f(x) = c),$$
$M$ contains every constant rational function; thus, $X_M$ contains all the constant functions. From the Stone-Weierstrass theorem we get that $X_M = \C(K/_M)$.

Since $X_M = \C(K/_M)$ is a separable space, $K/_M$ is metrizable compact.
\end{proof}
\section{Properties of sets}\
Let us consider a situation when we have a normed linear space $X$ and we want to recognize, whether a given set $A\subset X$ has a property $(P)$. For every separable subspace $V_0\subset X$ we want to find a closed separable subspace $V\supset V_0$ such that $A$ has the property $(P)$ in $X$ if and only if $A\cap V$ has the property $(P)$ in the subspace $V$.

Using the technic of elementary submodels, it is enough to show that for a suitable elementary submodel $M$ (dependent only on the space $X$ and perhaps also on the set $A$), the set $A$ has the property $(P)$ if and only if $A\cap X_M$ has the property $(P)$ in $X_M$.

Let us prove the results for properties ``to be dense'' and ``to have empty interior''.

\begin{prop}\label{pDense}\pp
Let $\langle X,\rho\rangle$ be a metric space and $A, S\subset X$. \ppXD[$A$]{$S$}
$$\Int_S{(A\cap S)} \neq \emptyset \iff \Int_{S\cap X_M}{(A\cap S\cap X_M)} \neq \emptyset,$$
$$A\cap S\text{ is dense in }S \iff A\cap S\cap X_M\text{ is dense in }S\cap X_M.$$
\end{prop}
\begin{proof} \letsMX, $A$ and $S$. According to the proposition \ref{pCountableSubset} we can see that $A^C\in M$ whenever $A\in M$. Since $A$ is dense in $X$ if and only if $A^C$ has empty interior in $X$, it is enough to show the first equivalence.

If $A\cap S$ has nonempty interior in $S$, then there exists a ball in $S$, which is a subset of $A\cap S$. Thus
$$(*)\qquad(\exists x\in S)(\exists r\in\er_+)(\forall y\in S)(z\in U(x,r)\rightarrow z\in A).$$
In the preceeding formula we use shortcut $y \in U(x,r)$, which stands for $y\in X \wedge \rho(y,x) <_\er r$. Free variables in the preceeding formula are $\er_+,X,\rho, <_{\er}, A, S$. Those are contained in $M$. This allows us to use the elementarity of $M$. Thus we find $x\in S\cap M$ and $r\in\er_+\cap M$ such that $((\forall y\in S)(z\in U(x,r)\rightarrow z\in A))^M$. Using the elementarity again, $U(x,r)\cap S$ is a subset of $A\cap S$. Consequently, $U(x,r)\cap S\cap X_M\subset A\cap S\cap X_M$. Since $x\in U(x,r)\cap S\cap X_M$, we have prooved that $A\cap S\cap X_M$ contains a nonempty open set in $S\cap X_M$.

Conversely, let us assume that $\Int_{S\cap X_M}{(A\cap S\cap X_M)}\neq \emptyset$. Then
$$(\exists x\in S\cap X_M)(\exists r\in\er_+)(U(x,r)\cap S\cap X_M\subset A\cap S).$$
Let us take $q\in(0,\frac{1}{2}r)\cap \qe_+$ and $x_0\in X\cap M$ such that $\rho(x,x_0) < q$. Then 
$$(U(x_0,q)\cap S\cap X_M)\subset (U(x,r)\cap S\cap X_M)\subset A\cap S.$$
For taken $x_0$ and $q$ holds $U(x_0,q)\cap S\cap M\subset A\cap S$. This can be written as
$$(\forall y\in S\cap M)\; (\rho(y,x_0) < q \rightarrow y\in A\cap S).$$
Therefore, using the absoluteness of
$$(*)\qquad(\forall y\in S)\; (\rho(y,x_0) < q \rightarrow y\in A\cap S),$$
we can see that $U(x_0,q)\cap S\subset A\cap S$. But the point $x$ is in $U(x_0,q)\cap S$. Consequently, $\Int_{S}{(A\cap S)}\neq \emptyset$.
\end{proof}

Another set property, which is separably determined, is ``to be nowhere dense''.

\begin{prop}\label{pNowhereDense}\pp
Let $\langle X,\rho\rangle$ be a metric space, $G\subset X$ an open set and $A\subset X$. \ppXD[$A$]{$G$}
$$A\cap G\text{ is nowhere dense in }G \iff A\cap G\cap X_M\text{ is nowhere dense in }G\cap X_M.$$
\end{prop}
\begin{proof}\letsMX, $A$ and $G$. According to the proposition \ref{pCountableSubset}, $C\cap B\in M$ whenever $C,B\in M$. It is well known, that $E\subset G$ is nowhere dense in $G$ \ifff it is nowhere dense in $X$ (see [\ref{kuratowski}, page 71]). Consequently, it is enough to prove the proposition for $G = X$.

It is well known, that set $A$ is nowhere dense in a metric space $X$ if and only if the following formula holds:
$$\forall x\in X\; \forall r\in\er_+\; \exists y\in X\; \exists s\in\er_+\; (U(y,s)\;\subset\;U(x,r)\setminus A).$$
It is easy to check that this is equivalent to the following formula:
\begin{equation}\label{eq:flejedna}
(*)\qquad\forall x\in X\; \forall r\in\qe_+\; \exists y\in X\; \exists s\in\qe_+\; (U(y,s)\;\subset\;U(x,r)\setminus A).
\end{equation}
All the free variables in the preceeding formula are elements of $M$.

Let us prove the implication from the right to the left first. If $A$ is not nowhere dense in $X$, then
$$(*)\qquad\exists x\in X\; \exists r\in\qe_+\; \forall y\in X\; \forall s\in\qe_+\; (U(y,s)\;\nsubseteq\;U(x,r)\setminus A).$$
Using the elementarity of $M$ there exists $x\in X\cap M$ and $r\in\qe_+$ such that:
\begin{equation}\label{eq:fledva}
\forall y\in X\; \forall s\in\qe_+ \; (U(y,s)\nsubseteq U(x,r)\setminus A).
\end{equation}
Choose an arbitrary $y\in X_M$, $s\in\qe_+$ and find such $y_0\in X\cap M$ that $\rho(y,y_0) < \frac{1}{2}s$. Then $U(y_0,\frac{1}{2}s)\subset U(y,s)$. From the validity of \eqref{eq:fledva},
$$(*)\qquad(\exists z\in X)\,(z\in \; U(y_0,\tfrac{1}{2}s)\;\setminus (U(x,r)\setminus A)).$$
Using elementarity of $M$ we may fix $z\in X\cap M$ satisfying the formula above. Thus, for given $y\in X_M$ and $s\in \qe_+$ we have found $z\in X\cap M$ satisfying
$$z\in\quad U(y_0,\tfrac{1}{2}s)\;\setminus (U(x,r)\setminus A)\quad\subset\quad U(y,s)\setminus (U(x,r)\cap X_M \setminus A).$$
Consequently,
$$U(y,s)\cap X_M \nsubseteq (U(x,r)\cap X_M) \setminus A.$$
The negation of \eqref{eq:flejedna} holds in $X_M$; thus, $A\cap X_M$ is not nowhere dense in $X_M$.

For the proof of converse implication, let $A$ be nowhere dense in $X$. Choose an arbitrary $x\in X_M$ and $r\in\qe_+$. Let us find  $x_0\in X\cap M$ satisfying $\rho(x,x_0) < \frac{1}{2}r$. Then $U(x_0,\frac{1}{2}r) \subset U(x,r)$. For given point $x_0$ and number $\frac{1}{2}r$ find $y\in X$ and $s\in \qe_+$ from the formula \eqref{eq:flejedna}. Using elementarity of $M$ we may assume that $y\in X\cap M$. Consequently,
$$U(y,s)\;\subset\; U(x_0,\tfrac{1}{2}r)\setminus A\; \subset\; U(x,r) \setminus A.$$
The formula \eqref{eq:flejedna} is satisfied in $X_M$; thus, $A\cap X_M$ is nowhere dense in $X_M$.
\end{proof}

Natural question is, how it is with the property ``to be meager''. One implication is simple.

\begin{prop}\label{pMeager}\pp
Let $X$ be a metric space. \ppXA:
$$A\text{ is meager in }X \rightarrow A\cap X_M\text{ is meager in }X_M.$$
\end{prop}
\begin{proof}
\letsMXA Let us have a family of nowhere dense sets $\{R_n\}_{n\in\omega}$ such that $A\subset\bigcup_{n\in\omega}R_n$.

Then
\begin{align*}
(*)\qquad(\exists \varphi)( & \varphi \text{ is a function with }\dom\varphi = \omega, \varphi(n)\text{ are nowhere dense subsets of $X$}\\
& \text{for every }n\in\omega,\; A\subset\bigcup_{n\in\omega}\varphi(n)).
\end{align*}

From the elementarity of $M$ we may take such a $\varphi\in M$. Consequently (using the proposition \ref{pRngInM}), $\varphi(n)\in M$ for every $n\in\omega$.

From the proposition \ref{pNowhereDense} we get, that for any $n\in\omega$ the set $\varphi(n)\cap X_M$ is nowhere dense in $X_M$. Besides that, $A\cap X_M\subset \bigcup_{n\in\omega}(\varphi(n)\cap X_M)$. Therefore, $A\cap X_M$ is meager in $X_M$. 
\end{proof}

For the converse implication of the preceeding proposition, we need to add some assumptions. Let us first recall what it means to be somewhere meager.

\begin{defin}
Let $X$ be a metric space and $A\subset X$. If there are $x\in X$ and $r > 0$ such that $U(x,r)\cap A$ is meager in $X$, we say that \emph{$A$ is somewhere meager in $X$}.
\end{defin}

We will need the following easy well-known fact.

\begin{lemma}\label{lResidual}
Let $X$ be a complete metric space and let $A\subset X$ have the Baire property. Then
$$X\setminus A \mbox{ is not meager } \iff A \mbox{ is somewhere meager in }X.$$
\end{lemma}

With the help of the preceeding lemma we are able to prove the converse implication of proposition \ref{pMeager}. First, we need to get the result for the properties ``Baire property'' and ``to be somewhere meager''.

\begin{prop}\label{pMeagerLoc}\pp
Let $X$ be a metric space. \ppXA
$$A\text{ is somewhere meager in }X \rightarrow A\cap X_M\text{ is somewhere meager in }X_M.$$
\end{prop}
\begin{proof}
\letsMX, $A$ and assume that $A$ is somewhere meager. With the use of the propositions \ref{pCountableSubset} and \ref{pUxrInM}, $U(x,r)\in M$ whenever $x\in X\cap M$, $r\in \er_+\cap M$ and $C\cap B\in M$ whenever $C,B\in M$.

Because $A$ is somewhere meager, the following formula holds:
$$(*)\qquad(\exists x\in X)(\exists r\in\er_+)(U(x,r)\cap A\text{ is meager in }X).$$
From the elementarity of $M$, we can find $x\in X\cap M$ and $r\in \er_+\cap M$ such that $U(x,r)\cap A$ is meager in $X$. Using the proposition \ref{pMeager} and the fact that $U(x,r)\cap A\in M$, we can see $U(x,r)\cap A\cap X_M$ is meager in $X_M$.
\end{proof}
\begin{prop}\label{pBaire}\pp
Let $X$ be a metric space. \ppXA:
$$A\text{ has Baire property in }X \rightarrow A\cap X_M\text{ has Baire property in }X_M.$$
\end{prop}
\begin{proof}\letsMX, $A$ and assume that $A$ has Baire property.
Then
$$(*)\qquad(\exists D)(\exists P)(D\text{ is a }G_\delta\text{ subset in $X$, }P\text{ is meager subset in $X$, }A = D\cup P).$$
Using the elementarity of $M$, we may take such $D,P\in M$. With the use of the proposition \ref{pMeager}, $P\cap X_M$ is meager in $X_M$.
Consequently, $A\cap X_M$ is union of the $G_\delta$ set $D\cap X_M$ and the meager set $P\cap X_M$.
\end{proof}

Finally, the converse of the proposition \ref{pMeager} can be proven under additional assumptions.

\begin{thm}\label{tMeager}\pp
 Let $X$ be a complete metric space, $G\subset X$ an open set and $A\subset X$ set with Baire property. \ppXD[$G$]{$A$}
 $$A\cap G\text{ is meager in }G \leftrightarrow A\cap G\cap X_M\text{ is meager in }G\cap X_M,$$
 $$A\cap G\text{ is residual in }G \leftrightarrow A\cap G\cap X_M\text{ is residual in }G\cap X_M.$$
\end{thm}
\begin{proof}\letsMX, $A$, $G$. According to the proposition \ref{pCountableSubset} it is true, that $B\cap C\in M$ and $B^C\in M$ whenever $B,C\in M$. 
It is well known that a set $D\subset G$ is meager in $X$ if and only if it is meager in $G$ (see [\ref{kuratowski}, page 83]). Thus, it is sufficient to prove the first equivalence for $G = X$.

The implication from the left to the right follows from the proposition \ref{pMeager}. For the converse implication, let us assume that $A$ is not meager in $X$. Consequently, using lemma \ref{lResidual}, $A^C$ is somewhere meager in $X$. Thus, according to the proposition \ref{pMeagerLoc}, $A^C\cap X_M$ is somewhere meager in $X_M$. Then from the propositions \ref{pBaire} and \ref{lResidual} we get, that $A\cap X_M$ is not meager in $X_M$.
\end{proof}

Let us find out, whether the property of sets ``to be porous'' is separably determined. When talking about porosity, we will use the following definition from \cite{zajicPory}.

\begin{defin}
Let $X$ be a metric space, $A\subset X$, $x\in X$ and $R > 0$. Then we define $\gamma(x,R,A)$ as the supremum of all $r\geq 0$ for which there exists $z\in X$ such that $U(z,r)\subset U(x,R)\setminus A$.\\
Further, we define the {\em upper porosity of $A$ at $x$ in the space $X$} as
$$\ov{p}_X(A,x) := 2\limsup_{R\to 0^+}\frac{\gamma(x,R,A)}{R}$$
and the {\em lower porosity of $A$ at $x$ in the space $X$} as
$$\underline{p}_X(A,x) := 2\liminf_{R\to 0^+}\frac{\gamma(x,R,A)}{R}$$

When it is clear which space $X$ we mean, then we often say {\em upper (lower) porosity of $A$ at $x$} and write $\ov{p}(A,x)$ ($\underline{p}(A,x)$).

We say that $A$ is \emph{upper porous} (\emph{lower porous}, {\em $c$-upper porous}, {\em $c$-lower porous}) at $x$ if $\ov{p}(A,x) > 0$ ($\underline{p}(A,x) > 0$, $\ov{p}(A,x)\geq c$, $\underline{p}(A,x)\geq c$).

We say that $A$ is \emph{upper porous} (\emph{lower porous}, {\em $c$-upper porous}, {\em $c$-lower porous}) if $A$ is upper porous (lower porous, $c$-upper porous, $c$-lower porous) at each $y\in A$. We say that $A$ is $\sigma$\emph{-upper (lower) porous} if it is a countable union of upper (lower) porous sets.
\end{defin}

\begin{defin}
 Let $\langle X,\rho\rangle$ be a metric space and $A\subset X$. Then by $d(\cdot,A)$ we understand the mapping which maps every $x\in X$ to $d(x,A):=\inf\{\rho(x,a);\;a\in A\}$.
\end{defin}

The following lemma is probably well known, but I didn't find any reference.

\begin{lemma}\label{lIffPorous}
 Let $\langle X, \rho\rangle$ be a metric space, $A\subset X$ and $x\in A$. Let us denote
 $$p_1(A,x) := \limsup_{R\rightarrow 0^+}\sup_{u\in U(x,R)}\frac{\d(u,A)}{R}\quad\text{and}\quad p_2(A,x) := \liminf_{R\rightarrow 0^+}\sup_{u\in U(x,R)}\frac{\d(u,A)}{R}.$$
 Then  $p_1(A,x) \leq \ov{p}(A,x) \leq 2 p_1(A,x)$ and $p_2(A,x) \leq \underline{p}(A,x) \leq 2 p_2(A,x)$.
\end{lemma}
\begin{proof} In order to show $\ov{p}(A,x)\leq 2 p_1(A,x)$ and $\underline{p}(A,x)\leq 2 p_2(A,x)$, it is sufficient to prove that $\gamma(x,R,A) \leq \sup_{u\in U(x,R)}\d(u,A)$ for every $R > 0$. Choose some $R>0$, $r\geq 0$ and $z\in X$ satisfying $U(z,r)\subset U(x,R)\setminus A$. We would like to find $u\in U(x,R)$ such that $r \leq \d(u,A)$. But it is easy to check that $u = z$ satisfies those conditions.

  Now we will prove that $\ov{p}(A,x) \geq p_1(A,x)$ and $\underline{p}(A,x) \geq p_2(A,x)$. Take an arbitrary $R > 0$, $u\in U(x,R)$ and notice that then $\d(u,A)\leq\gamma(x,2R,A)$.\\  
  Really, put $r = \d(u,A)$ and $z = u$. Then for every $y\in U(z,r)$ is
    $$\rho(u, y) = \rho(z, y) < r = \d(u,A),$$
  so $y\notin A$. Besides that (using the fact that $r = \d(u,A) < R$, since $x\in A$ and $u\in U(x,R)$),
    $$\rho(y, x) \leq \rho(y,z) + \rho(z,x) < r + R < 2R.$$
  Thus, $U(z,r)\subset U(x,2R)\setminus A$ and $\d(u,A)\leq\gamma(x,2R,A)$.
    
  An immediate consequence is, that
  $$2\limsup_{R\to 0^+}\frac{\gamma(x,2R,A)}{2R}\geq p_1(A,x),\qquad2\liminf_{R\to 0^+}\frac{\gamma(x,2R,A)}{2R}\geq p_2(A,x).$$
 Now it is easy to check that also $\ov{p}(A,x)\geq p_1(A,x)$ and $\underline{p}(A,x)\geq p_2(A,x)$.                                      
\end{proof}

The following two propositions show that the first implication about porous sets holds.

\begin{prop}\label{pUpPorous}\pp
 Let $\langle X,\rho\rangle$ be a metric space. \ppXA
$$A\text{ is not upper porous in }X \rightarrow A\cap X_M\text{ is not upper porous in }X_M.$$
\end{prop}
\begin{proof}
 \letsMX, $A$. The set $A$ is upper porous in $X$ \ifff the following formula holds:
$$\forall x\in A\;\exists m\in\qe_+\;\forall R_0 > 0\; \exists R\in (0,R_0)\;\gamma(x,R,A) > Rm.$$
This formula is equivalent to the following one:
$$\forall x\in A\;\exists m\in\qe_+\; \forall R_0 > 0\; \exists R\in (0,R_0)\; \exists r > Rm\; \exists z\in X\quad U(z,r)\subset U(x,R)\setminus A.$$

Let us notice that this formula is equivalent to the formula, where we take only rational numbers $R_0, R$ and $r$. It is obvious that we may consider only rational numbers $R_0$. Let us take an arbitrary $x\in A$, $m\in\qe_+$ from the formula above and $R_0\in\qe_+$. Then
$$\exists R\in (0,R_0)\; \exists r > Rm\;\exists z\in X\quad U(z,r)\subset U(x,R)\setminus A.$$
Fix $R\in(0,R_0)$, $r > Rm$ and $z\in X$ from the formula above. If we take a rational number $R_q$ from the interval  $(R,\min\{R_0,\frac rm\})$, then $U(z,r)\subset U(x,R_q)\setminus A$. Thus, $R$ may be without loss of generality considered to be rational. Having now rational number $R\in(0,R_0)$, real number $r > Rm$ and $z\in X$ such that $U(z,r)\subset U(x,R)\setminus A$, let us take a rational number $r_q$ from the interval $(Rm,r)$. Then $U(z,r_q)\subset U(x,R)\setminus A$. Consequently, the number $r$ may be without loss of generality considered to be rational.

We have seen that $A$ is not upper porous in $X$ \ifff the following formula holds:
\begin{equation}\label{eq:notPorosity}\begin{split}
(*)\qquad \exists x\in A\;\forall m\in\qe_+\;\exists R_0\in\qe_+\; \forall R\in (0,R_0)\cap\qe_+\; \forall r\in(Rm,\infty)\cap\qe_+ \\ \forall z\in X\;\;
U(z,r)\nsubseteq U(x,R)\setminus A.
\end{split}\end{equation}

Thus, when $A$ is not upper porous in $X$ we are able to find a point $x\in A$ from \eqref{eq:notPorosity}. Using the elementarity of $M$, we may assume that $x\in M$. Now fix $m\in\qe_+$ and find $R_0\in\qe_+$ from the formula \eqref{eq:notPorosity}. Fix $R\in (0,R_0)\cap\qe_+$, $r\in(Rm,\infty)\cap\qe_+$ and $z\in X_M$. Then find $r'\in(Rm,r)\cap\qe$ and $z_0\in X\cap M$ such that $\rho(z,z_0) < r - r'$. Thus, $U(z_0,r')\subset U(z,r)$. Then the following holds:
$$(*)\qquad(\exists y\in X)\;(y\in U(z_0,r') \; \setminus \; (U(x,R)\setminus A)).$$
For $r'$ and $z_0$ we are able to find (using the elementarity of $M$) point $y\in M$ such that
$$y\in U(z_0,r') \; \setminus \; (U(x,R)\setminus A)\quad \subset \quad U(z,r) \; \setminus \; (U(x,R)\setminus A).$$

Consequently, the formula \eqref{eq:notPorosity} is satisfied in $X_M$ and the set $A\cap X_M$ is not upper porous in $X_M$.
\end{proof}

\begin{prop}\label{pLowerPorous}\pp
 Let $X$ be a metric space. \ppXA
$$A\text{ is not lower porous in }X \rightarrow A\cap X_M\text{ is not lower porous in }X_M.$$
\end{prop}
\begin{proof}
 \letsMX, $A$. If $A$ is not lower porous, then similarly as in the proof of the proposition \ref{pUpPorous}, the following formula holds:
\begin{equation}\label{eq:notLowerPorosity}\begin{split}
(*)\qquad\exists x\in A\;\forall m\in\qe_+\;\forall R_0\in\qe_+\; \exists R\in (0,R_0)\; \forall r\in(Rm,\infty)\cap\qe_+ \\ \forall z\in X\;\; U(z,r)\nsubseteq U(x,R)\setminus A.
\end{split}\end{equation}
Using the elementarity of $M$, let us take $x\in A\cap M$ from the formula above. Then fix $m,R_0\in\qe_+$ and find $R\in(0,R_0)$ such that
$$\forall r\in(Rm,\infty)\cap\qe_+ \; \forall z\in X\;\;U(z,r)\nsubseteq U(x,R)\setminus A.$$
Using the elementarity of $M$ we may assume that $R\in M$. Now let us choose an arbitrary $r\in(Rm,\infty)\cap\qe_+$ and $z\in X_M$. Then find $r'\in (Rm,r)\cap\qe$ and $z_0\in U(z,r-r')\cap M$. Thus, $U(z_0,r')\subset U(z,r)$. Then the following holds:
$$(*)\qquad(\exists y\in X)\;(y\in U(z_0,r') \; \setminus \; (U(x,R)\setminus A)).$$
For $r'$ and $z_0$ we are able to find (using the elementarity of $M$) point $y\in M$ such that $y\in U(z_0,r') \setminus(U(x,R)\setminus A)$. Consequently,
$$X_M\cap U(z,r) \nsubseteq U(x,R) \setminus A.$$
Thus, the formula \eqref{eq:notLowerPorosity} is satisfied in $X_M$ and $A\cap X_M$ is not lower porous in $X_M$.
\end{proof}

To see that the converse implication holds we will follow the ideas presented in [\ref{tiser}, page 42]. The following result is proven  there for a rich family of subspaces (in the case that $X$ is a Banach space). We show the proof for spaces constructed from elementary submodels (which holds even in the case of metric spaces).

\begin{lemma}\label{lsupFinM}\pp
 Let $\langle X,\rho\rangle$ be a metric space and $f:X\to\er$ a function. \ppXJ[$f$] for every $R > 0$ and $x\in X_M$:
  $$\sup_{u\in U(x,R)}f(u) = \sup_{u\in U(x,R)\cap X_M}f(u).$$
\end{lemma}
\begin{proof}
  \letsMX, $f$. Fix $x\in X_M$ and $R > 0$. We would like to verify that $\sup_{u\in U(x,R)}f(u) \leq \sup_{u\in U(x,R)\cap X_M}f(u)$ (the other inequality is obvious). For this purpose, let us take an arbitrary $S\in\qe_+$ satisfying $S < \sup_{u\in U(x,R)}f(u)$. Then there exists $u\in U(x,R)$ such that $S < f(u)$. Now, find rational numbers $R_q, \eps\in\qe_+$ such that $R_q < R$ and $\rho(u,x) < R_q - \eps$. Let us take some $x_0\in U(x,\frac \eps 2)\cap M$. Then $u\in U(x_0,R_q - \frac \eps 2)$ and using the absoluteness of the formula
  $$(*)\qquad\exists u\in X:\rho(u,x_0) < R_q - \tfrac \eps 2\;\wedge\;S < f(u),$$
  we get the existence of $u\in U(x_0,R_q - \frac \eps 2)\cap M \subset U(x,R)\cap M$ such that $S < f(u)$. Consequently, $S < \sup_{u\in U(x,R)\cap X_M}f(u)$.
\end{proof}

\begin{prop}\label{pPorous}\pp
 Let $X$ be a metric space. \ppXD[$A\subset X$]{$d(\cdot,A)$} for every $x\in A\cap X_M$
$$A\text{ is lower porous at }x \rightarrow A\cap X_M\text{ is lower porous at }x\text{ in the space $X_M$},$$
$$A\text{ is upper porous at }x \rightarrow A\cap X_M\text{ is upper porous at }x\text{ in the space $X_M$}.$$
\end{prop}
\begin{proof}
 \letsMX, $A$, $d(\cdot,A)$ and fix some $x\in A\cap X_M$ such that $A$ is $c$-upper porous at $x$ for some rational $c > 0$. Thus, from the lemma \ref{lIffPorous} and \ref{lsupFinM} it follows that
  \begin{eqnarray*}
    c\leq\ov{p}_X(A,x) & &\leq 2\limsup_{R\rightarrow 0^+}\sup_{u\in U(x,R)}\frac{\d(u,A)}{R} = 2\limsup_{R\rightarrow 0^+}\sup_{u\in U(x,R)\cap X_M}\frac{\d(u,A)}{R} \\
    & &\leq 2\limsup_{R\rightarrow 0^+}\sup_{u\in U(x,R)\cap X_M}\frac{\d(u,A\cap X_M)}{R} \leq 2\ov{p}_{X_M}(A\cap X_M,x).
  \end{eqnarray*}
 Consequently, $A\cap X_M$ is $\tfrac c2$-upper porous in the space $X_M$. The result for lower porosity follows similarly.
\end{proof}

\begin{cor}\label{tPorous}\pp
 Let $X$ be a metric space. \ppXD[$A\subset X$]{$\d(\cdot,A)$}
$$A\text{ is lower porous in }X \iff A\cap X_M\text{ is lower porous in }X_M,$$
$$A\text{ is upper porous in }X \iff A\cap X_M\text{ is upper porous in }X_M,$$
$$A\text{ is $\sigma$-lower porous in }X \rightarrow A\cap X_M\text{ is $\sigma$-lower porous in }X_M,$$
$$A\text{ is $\sigma$-upper porous in }X \rightarrow A\cap X_M\text{ is $\sigma$-upper porous in }X_M.$$
\end{cor}
\begin{proof}\letsMX, $A$, $d(\cdot,A)$. Then the porosity results follow from the propositions \ref{pUpPorous}, \ref{pLowerPorous} and \ref{pPorous}. The $\sigma$-porosity results are then obtained similarly as in the proof of meagerness \ref{pMeager} using the absoluteness of the following two formulas
\begin{align*}
(*)\qquad(\exists \varphi)( & \varphi \text{ is function with }\dom\varphi = \omega, \varphi(n)\text{ are lower porous subsets of $X$}\\
& \text{for every }n\in\omega,\;  A\subset\bigcup_{n\in\omega}\varphi(n)).
\end{align*}
\begin{align*}
(*)\qquad(\exists \varphi)( & \varphi \text{ is function with }\dom\varphi = \omega, \varphi(n)\text{ are upper porous subsets of $X$}\\
& \text{for every }n\in\omega,\;  A\subset\bigcup_{n\in\omega}\varphi(n)).
\end{align*}
\end{proof}

It remains unknown to the author whether the converse implication of the preceeding results about $\sigma$-porosity holds as well.
\section{Properties of functions}\
Let us consider a situation when we have a normed linear space $X$ and we have a function $f$ defined on $X$. The aim of this section is to say which properties $(P)$ of the function $f$ are ``separably determined''. To be more concrete, we want to find a closed separable subspace $X_M$ such that for every $x\in X_M$ it is true that:
$$f\text{ has the property }(P)\text{ at }x\iff f\upharpoonright_{X_M}\text{ has the property }(P)\text{ at }x.$$

Using the technic of elementary submodels it is possible to combine those results about functions with the ones about sets.

First of the function properties we are interested in is the continuity.
\begin{defin}
 Let $\langle X,\rho\rangle$ and $\langle Y,\sigma\rangle$ be metric spaces, $G\subset X$ open subset and $f:G\to Y$ a function. Then we denote by $C(f)$ the set of points where $f$ is continuous.
\end{defin}

\begin{thm}\label{tContinuous}\pp
 Let $\langle X,\rho\rangle$ and $\langle Y,\sigma\rangle$ be metric spaces, $G\subset X$ open subset and $f:G\to Y$ a function. \ppXfJ{$Y$} $C(f)\in M$ and for every $x\in X_M\cap G$:
 $$f\text{ is continuous at }x\iff f\upharpoonright_{X_M}\text{ is continuous at }x.$$
\end{thm}
\begin{proof}
 \letsMX, $Y$, $f$. Then it is true that $G\in M$, since $G=\dom(f)$. $C(f)$ is an object uniquely defined by the formula
 $$(*)\qquad(\exists C)(\forall z)(z\in C \leftrightarrow z\in G \wedge f\text{ is continuous at }z);$$
 hence $C(f)\in M$. Let us prove the equivalece now. The implication from the left to right holds for an arbitrary subspace of $X$.
 If the function $f$ is not continuous at $x\in X_M\cap G$, then we find $k\in\omega$ such that the following formula holds:
\begin{equation}\label{eq:nespojitost}
\forall n\in\omega\;\exists y,z\in G: \left(y,z\in U(x,\tfrac{1}{n})\;\;\wedge\;\;\sigma(f(y),f(z))>\tfrac{1}{k}\right).
\end{equation}
Fix $n\in\omega$ and $x_0\in U(x,\frac{1}{2n})\cap M$. Then $U(x_0,\frac{1}{2n})$ is open set containing $x$, so there exists $l\in\omega$ such that $U(x,\frac{1}{l})$ is a subset of $U(x_0,\frac{1}{2n})$.
 According to the formula \eqref{eq:nespojitost} there are $y,z\in G$ satisfying
$$y,z\in U(x,\tfrac{1}{l})\;\;\wedge\;\;\sigma(f(y),f(z))>\tfrac{1}{k}.$$
Consequently, the following formula is satisfied:
$$(*)\qquad\exists y,z\in G: \left( y,z\in U(x_0,\tfrac{1}{2n})\;\;\wedge\;\;\sigma(f(y),f(z))>\tfrac{1}{k}\right).$$
All the free variables in this formula are in $M$, so using the elementarity of $M$ and the fact that $U(x_0,\frac{1}{2n})\subset U(x,\frac{1}{n})$, we get points $y,z\in G\cap M$ such that
\begin{equation}\label{eq:nespoj}
	y,z\in U(x,\tfrac{1}{n})\;\;\wedge\;\;\sigma(f(y),f(z))>\tfrac{1}{k}.
\end{equation}

We have just shown that for an arbitrary $n\in\omega$ we are able to find points $y,z\in G\cap M$ satisfying \eqref{eq:nespoj}. Consequently, the function $f\upharpoonright_{X_M}$ is not continuous at $x$.
\end{proof}

 Having proven that the property $(P)$ of the function $f$ (continuity in this case) is separably determined, we get that for a set $A:=\{x: f\text{ has the property }(P)\text{ at }x\}$ the following holds:
$$A\cap X_M = \{x: f\upharpoonright_{X_M}\text{ has the property }(P)\text{ at }x\}.$$
Combinig this result with the result from the previous section we get the existence of a closed separable subspace $X_M$ such that
\begin{align*}
	&\{x: f\text{ has the property }(P)\text{ at }x\}\text{ is dense in }X \iff \\
	&\{x: f\upharpoonright_{X_M}\text{ has the property }(P)\text{ at }x\}\text{ is dense in }X_M.
\end{align*}

 Thus, an immediate consequence of the preceeding theorem and results about separably determined set properties is the following.

\begin{cor}\pp
 Let $X$ and $Y$ be metric spaces, $G\subset X$ open subset and $f:G\to Y$ a function. Let $X$ be complete. \ppXD[$Y$]{$f$}
 $$C(f)\text{ is dense in }G \leftrightarrow C(f\upharpoonright_{X_M})\text{ is dense in }G\cap X_M,$$
 $$C(f)\text{ is nowhere dense in }G \leftrightarrow C(f\upharpoonright_{X_M})\text{ is nowhere dense in }G\cap X_M,$$
 $$C(f)\text{ is meager in }G \leftrightarrow C(f\upharpoonright_{X_M})\text{ is meager in }G\cap X_M,$$
 $$C(f)\text{ is residual }G \leftrightarrow C(f\upharpoonright_{X_M})\text{ is residual in }G\cap X_M,$$
 $$C(f)\text{ is not upper porous in }X \leftrightarrow C(f\upharpoonright_{X_M})\text{ is not upper porous in }X_M,$$
 $$C(f)\text{ is not lower porous in }X \leftrightarrow C(f\upharpoonright_{X_M})\text{ is not lower porous in }X_M.$$
\end{cor}
\begin{proof}
 \letsMX, $Y$, $f$. Then $G\in M$, because $G = \dom(f)$. It is well known, that $C(f)$ is a $G_\delta$ set [\ref{kuratowski}, page 207-208]. From the preceeding theorem, $C(f)\cap X_M = C(f\upharpoonright_{X_M})$. Therefore we get the wanted result as an immediate consequence of propositions \ref{pDense}, \ref{pNowhereDense}, \ref{pUpPorous}, \ref{pLowerPorous} and theorems \ref{tMeager}, \ref{tContinuous}.
\end{proof}

Next function property we examine is the lower (upper) semicontinuousity. Let us recall the definition in metric spaces.

\begin{defin}
Let $X$ be a metric space, $G\subset X$ open subset, $f:G\to [-\infty,\infty]$ a function and $x\in G$. If for every sequence $\{x_n\}_{n\in\omega}\subset G$, $x_n \to x$ it is true that $$\liminf_{n\to\infty}f(x_n)\geq f(x),$$ 
then we say that $f$ is \emph{lower semicontinuous} (\emph{lsc}) at $x$.

If the function ($-f$) is \emph{lsc} at $x$, we say that $f$ is \emph{upper semicontinuous} (\emph{usc}) at $x$.
\end{defin}

The following lemma will be used in the proposition saying that the lower (upper) semicontinuity is separably determined property.
\begin{lemma}\label{lLsc}
Let $X$ be a metric space, $G\subset X$ open subset, $f:G\to [-\infty,\infty]$ a function and $x\in G$.
Then $f$ is $lsc$ at $x$ \ifff for every $c\in \qe\cap (-\infty, f(x))$ there exists $n\in\omega$ such that $f\left[U(x,\frac{1}{n})\cap G\right]\subset (c,\infty]$.
\end{lemma}
\begin{proof} We may assume that $f(x) > -\infty$ (if $f(x) = - \infty$, then the lemma is obvious).\\
``$\Rightarrow$'' Suppose that there are number $c\in \qe\cap (-\infty, f(x))$ and sequence $\{x_n\}_{n\in\omega}\subset G$ such that $x_n\in U(x,\frac{1}{n})$, but $f(x_n)\leq c$. Then $x_n \to x$, but $\liminf_{n\to\infty}f(x_n) \leq c < f(x)$. Thus, $f$ is not $lsc$ at $x$.

``$\Leftarrow$'' First, let us assume that $f(x) < \infty$. Fix $\eps > 0$, $c\in \qe\cap (f(x) - \eps, f(x))$ and sequence $\{x_n\}_{n\in\omega}\subset G$, $x_n \to x$. Then there exists $k\in\omega$ such that $f\left[U(x,\frac{1}{k})\cap G\right]\subset (c,\infty]$. Next, there exists $n_0$ such that for every $n\geq n_0$ is $x_n\in U(x,\frac{1}{k})$. Then for every $n\geq n_0$ is $f(x_n) > c > f(x) - \eps$. Thus, $\liminf_{n\to\infty}f(x_n)\geq f(x) - \eps$. Because we have chosen an arbitrary $\eps > 0$, it is true that $\liminf_{n\to\infty}f(x_n)\geq f(x)$.

In the case that $f(x) = \infty$, we will fix $K\in\omega$, $c \in \qe\cap (K,\infty)$ and sequence $\{x_n\}_{n\in\omega}\subset G$, $x_n \to x$. Similarly as above it follows that $\liminf_{n\to\infty}f(x_n)\geq K$.
\end{proof}

\begin{prop}\label{tLsc}\pp
Let $X$ be a metric space, $G\subset X$ open subset and $f:G\to [-\infty,\infty]$ a function. \ppXf
$$f\text{ is }lsc\text{ at }x\iff f\upharpoonright_{X_M}\text{ is }lsc\text{ at }x.$$
\end{prop}
\begin{proof}
 Immediately from the definition it is obvious that the implication from the left to the right holds for any subspace of $X$. \letsMX, $f$ and assume that $f$ is not $lsc$ at $x\in X_M\cap G$. Then from the lemma \ref{lLsc} we get the existence of $c\in \qe\cap (-\infty, f(x))$ such that for every $n\in\omega$ exists $y\in U(x,\frac{1}{n})\cap G$ such that $f(y) \leq c$. Choose an arbitrary $n\in\omega$ and find $x_0\in U(x,\frac{1}{2n}) \cap M$. Then $U(x_0,\frac{1}{2n})\subset U(x,\frac{1}{n})$ is open set containing $x$, so there exists $l\in\omega$ such that $U(x,\frac{1}{l})\subset U(x_0,\frac{1}{2n})$. For such $l\in\omega$ there exists $y\in U(x,\frac{1}{l})\cap G$ such that $f(y) \leq c$. Consequently, the following formula holds:
 $$(*)\qquad\exists y\in U(x_0,\tfrac{1}{2n})\cap G:\quad f(y)\leq c.$$
 Using the elementarity of $M$ we find $y\in U(x_0,\frac{1}{2n})\cap G\cap M\subset U(x,\frac{1}{n})\cap G\cap M$ such that $f(y)\leq c$. For an arbitrary $n\in\omega$ we have found $y\in U(x,\frac{1}{n})\cap G\cap X_M$ such that $f(y)\leq c$. It follows from the lemma \ref{lLsc} that $f\upharpoonright_{X_M}$ is not $lsc$ at $x$.
\end{proof}

\begin{cor}\pp
 Let $X$ be a metric space, $G\subset X$ open subset and $f:G\to [-\infty,\infty]$ a function. Let us denote by $-$ the operation which maps every function $h:G\to [-\infty,\infty]$ to the function $-h$. \ppXfD{$-$}
 $$f\text{ is }usc\text{ at }x\iff f\upharpoonright_{X_M}\text{ is }usc\text{ at }x.$$
\end{cor}
\begin{proof}
 \letsMX, $f$, $-$. Then $-f\in M$, thus it is enough to use the preceeding proposition.
\end{proof}

The last function property examined in this article is the Fr\'echet differentiability. We will use the following definiton.

\begin{defin}
 Let $X$ and $Y$ be NLS, $G\subset X$ open subset, $f:G\to Y$ function and $x\in G$.
\begin{itemize}
	\item[(i)] If there exists continuous linear operator $A:X\to Y$ such that
						$$\lim_{u\to x}\frac{f(u)-f(x)-A(u-x)}{\|u-x\|} = 0,$$
						then we say that \emph{function $f$ is Fr\'echet differentiable at $x$}. The set of the points where $f$ is Fr\'echet differentiable we will denote by $D(f)$.
	\item[(ii)] For $c>0$, $\eps>0$ and $\delta>0$ we define $D(f,c,\eps,\delta)$ as the set of all points $x\in G$ satisfying
	$$\left\|\frac{f(y+tv)-f(y)}{t}-\frac{f(y)-f(y-hv)}{h}\right\|\leq\eps$$
	whenever
	\begin{align*}
	&v\in X,\quad \|v\|=1,\quad t>0,\quad h>0,\quad y\in U(x,\delta),\quad y - hv \in U(x,\delta),\\
	&y + tv \in U(x,\delta)\quad \text{ and }\quad \min(t,h)>c\|y-x\|.
	\end{align*}
\end{itemize}
\end{defin}

The following relationship between sets $D(f,c,\eps,\delta)$ and Fr\'echet differentiability is shown in \cite{zajicDif}.

\begin{lemma}\label{lZajic}
 Let $X$ be NLS, $G\subset X$ open subset and let $Y$ be a Banach space. Let $f:G\to Y$ be a function. Then $f$ is Fr\'echet differentiable at a point $x\in G$ \ifff $f$ is continuous at $x$ and $x\in\bigcap_{n\in\en}\bigcup_{k\in\en}D(f,\frac{1}{n},\frac{1}{n},\frac{1}{k})$.
\end{lemma}

Using this lemma, it is shown in \cite{zajicDif} that the property ``to be Fr\'echet differentiable'' is separably determined. Let us prove a similar result using the technic of elementary submodels.

\begin{thm}\label{tDif}\pp
 Let $X$ be a NLS, $G\subset X$ an open subset and $Y$ a Banach space. Let $f:G\to Y$ be a function. \ppXfJ{$Y$} $D(f)\in M$ and for every $x\in X_M\cap G$:
 $$f\text{ is Fr\'echet differentiable at }x \iff f\upharpoonright_{X_M}\text{ is Fr\'echet differentiable at }x.$$
\end{thm}
\begin{proof}
 \letsMX, $Y$, $f$. $D(f)$ is an object uniquely defined by the formula
 $$(*)\qquad(\exists D)(\forall z)(z\in D \leftrightarrow z\in D \wedge f\text{ is Fr\'echet differentiable at }z);$$ hence $D(f)\in M$.
 Fix a point $x\in X_M\cap G$. Then according to the theorem \ref{tContinuous} it is true that $f$ is continuous at $x$ \ifff $f\upharpoonright_{X_M}$ is continuous at $x$. Thus, using the lemma \ref{lZajic}, it is sufficient to check that
$$x\in\bigcap_{n\in\en}\bigcup_{k\in\en}D(f,\tfrac{1}{n},\tfrac{1}{n},\tfrac{1}{k}) \iff x\in\bigcap_{n\in\en}\bigcup_{k\in\en}D(f\upharpoonright_{X_M},\tfrac{1}{n},\tfrac{1}{n},\tfrac{1}{k}).$$
 The implication from the left to the right is obvious (it holds for every subspace of $X$). Conversely, let us assume that $x\notin\bigcap_{n\in\en}\bigcup_{k\in\en}D(f,\frac{1}{n},\frac{1}{n},\frac{1}{k})$. Fix $n\in\en$ satisfying 
 $x\notin\bigcup_{k\in\en}D(f,\frac{1}{n},\frac{1}{n},\frac{1}{k})$. Then for every $k\in\omega$ the following formula holds:
 $$\exists v\in X,\quad \|v\|=1,\quad \exists t,h>0,\quad \exists y\in X:$$
 \begin{equation*}
	\left(
	\begin{aligned}
	&y\in U(x,\tfrac{1}{k}),\quad y - hv \in U(x,\tfrac{1}{k}),\quad
	y + tv \in U(x,\tfrac{1}{k}), \\
	& \min(t,h)>\frac{1}{n}(\|y-x\| + 0), \quad \left\|\frac{f(y+tv)-f(y)}{t}-\frac{f(y)-f(y-hv)}{h}\right\| > \frac{1}{n}
	\end{aligned}
	\right).
 \end{equation*}
 Mark this formula with $(*)$.

 Let us take some $v,t,h$ and $y$ from the formula above and find $\eta\in\qe_+$ such that
 \begin{align*}
	&\|y - x\| < \frac{1}{k} - 2\eta,\quad \|y - hv - x\| < \frac{1}{k} - 2\eta,\\
	&\|y + tv - x\| < \frac{1}{k} - 2\eta,\quad \min(t,h) > \frac{1}{n}(\|y - x\| + 2\eta).
 \end{align*}
 Further, take $x_0\in U(x,\eta)\cap M$. Then the following holds:
 \begin{align*}
	&\|y - x_0\| \leq \|y - x\| + \|x-x_0\| < \frac{1}{k} - \eta,\quad \|y - hv - x_0\| < \frac{1}{k} - \eta,\\
	&\|y + tv - x_0\| < \frac{1}{k} - \eta,\quad \frac{1}{n}(\|y - x_0\| + \eta) \leq \frac{1}{n}(\|y - x\| + 2\eta) < \min(t,h).
 \end{align*}

 Using the elementarity of $M$ we get the existence of $v\in X\cap M$, $\|v\| = 1$, $t,h\in\er_+\cap M$ a $y\in X\cap M$ such that:
 \begin{align*}
	&y\in U(x_0,\tfrac{1}{k} - \eta)\subset U(x,\tfrac{1}{k}),\quad y - hv \in U(x_0,\tfrac{1}{k} - \eta)\subset U(x,\tfrac{1}{k}),\\
	&y + tv \in U(x_0,\tfrac{1}{k} - \eta)\subset U(x,\tfrac{1}{k}),\quad
	\min(t,h)> \frac{1}{n}(\|y - x_0\| + \eta) > \frac{1}{n}\|y-x\|,\\
	&\left\|\frac{f(y+kv)-f(y)}{k}-\frac{f(y)-f(y-hv)}{h}\right\| > \frac{1}{n}.
 \end{align*}
 Consequently, $x\notin\bigcap_{n\in\en}\bigcup_{k\in\en}D(f\upharpoonright_{X_M},\frac{1}{n},\frac{1}{n},\frac{1}{k})$.
\end{proof}

We would like to combine this result with the theorem \ref{tMeager}, saying that beeing a residual subset is separable determined property for sets with Baire property in complete metric spaces. The following result comes from \cite{zajicDif}.

\begin{thm}\label{tDifSet}
 Let $X$ be a normed linear space, $G\subset X$ open subset, and let $Y$ be a Banach space. Let $f:G\to Y$ be a function. Then $D(f)$ is an $F_{\sigma\delta}$ set.
\end{thm}

Using this result we immediately get the following corrolary (obviously, it is true even more, similarly as in the case of continuity).

\begin{cor}\label{cDif}\pp
 Let $X$, $Y$ be Banach spaces, $G\subset X$ open subset and $f:G\to Y$ a function. \ppXD[$Y$]{$f$}
$$D(f)\text{ is dense in }G \leftrightarrow D(f\upharpoonright_{X_M})\text{ is dense in }G\cap X_M,$$
$$D(f)\text{ is residual in }G \leftrightarrow D(f\upharpoonright_{X_M})\text{ is residual in }G\cap X_M.$$
\end{cor}
\section{Applications}
In this last section we show two applications of the theorems proven above. Both of them extend the validity of already known theorems to special nonseparable spaces. In the first case we will be interested in the result proven in [\ref{zajicDif2}; Proposition 3.3] by Zaj\'{\i}\v{c}ek for spaces with separable dual. The technic of elementary submodels will allow us to prove that the same theorem holds in general Asplund spaces. The second application will extend the result proven in [\ref{preis}; Theorem 4.8] from the spaces $\C(K)$ with $K$ is a countable compact and from subspaces of $c_0$ to the spaces $\C(K)$ with $K$ is a general scattered compact and to subspaces of $c_0(\Gamma)$ with possibly uncountable set $\Gamma$.

Separable reductions of the results mentioned above have already been examined using the technic of rich families (to remind the concept of rich families, see the section \ref{richFamilies}). In the first case Zaj\'{\i}\v{c}ek in [\ref{zajicDif2}; Theorem 5.2] achieved to prove only a weaker variant of the theorem in Asplund spaces. In the second case, the separable reduction to the subspaces of $c_0(\Gamma)$ easilly follows using the work of J.Lindenstrauss, D.Preiss and J.Ti\v{s}er [\ref{tiser}; Corrolary 5.6.2] and the result of Zaj\'{\i}\v{c}ek [\ref{zajicDif2}; Theorem 4.7]. However, to the author it is not known whether the extension to spaces $\C(K)$ with $K$ scattered compact has been proven anywhere else.

Let us introduce the first application now.

L. Zaj\'{\i}\v{c}ek proved in \cite{zajicDif2} theorem marked in this text as theorem \ref{tZajic}. This theorem was proven for spaces with separable dual. We will show how to use the technic of elementary submodels to get the same result for Asplund spaces.

In the following, if it is not said otherwise, $X$ will be a Banach space. The equality $X = X_1 \oplus \ldots \oplus X_n$ means that $X$ is the direct sum of non-trivial closed linear subspaces $X_1,\ldots,X_n$ and the corresponding projections $P_i:X\to X_i$ are continuous.

Recall that $X$ is an Asplund space if each continuous convex real valued function on $X$ is Fr\'echet differentiable at each point of $X$ except a first category set and that $X$ is Asplund space \ifff $Y^*$ is separable for every separable subspace $Y\subset X$.

We will need the following well-known fact (see [\ref{zajicDif2}]).

\begin{lemma}\label{lLines}
 Let $X$ be a Banach space, $0\neq u\in X$, and let $X = W\oplus\sspan\{u\}$. Then the mapping $w\in W\mapsto w+\er u\in X/\sspan\{u\}$ is a linear homeomorphism.
\end{lemma}

To formulate the result from \cite{zajicDif2}, we need the following definition.
\begin{defin}
 Let $f$ be a real valued function defined on an open subset $G$ of a Banach space $X$.
 \begin{itemize}
  \item[(i)] We say that $f$ is generically Fr\'echet differentiable on $G$ if the set $D(f)$ of points where $f$ is Fr\'echet differentiable is residual in $G$.
  \item[(ii)] We say that $f$ is strictly differentiable at $a\in G$ if there exists $x^*\in X^*$ such that
    $$\lim_{(x,y)\to(a,a), x\neq y} \frac{f(y)-f(x)-x^*(y-x)}{\|y-x\|} = 0.$$
  \item[(iii)] We say that $f$ is {\em essentially smooth} ({\em esm} for short) {\em on the line} $L = a + \er v$ (where $a\in X$, $0\neq v\in X$) if the function $\phi(t):=f(a + tv)$ is strictly differentiable at a.e. points of its domain. (Obviously, the definition is correct: it does not depend on the choice of $a$ and $v$).
  \item[(iv)] We say that line $L$ is parallel to $v$ (where $0\neq v\in X$), if there exists $a\in X$ such that $L = a + \er v$.
  \item[(v)] We say that $f$ is {\em essentially smooth on a generic line parallel to $0\neq v\in X$}, if $f$ is essentially smooth on all lines parallel to $v$, except a first category set of lines in the factor space $X/\sspan\{v\}$.
 \end{itemize}
\end{defin}

\begin{remark}
 Let $X$ be a NLS, $G\subset X$ open subset, $f:G\to \er$ function, $Y$ a subspace of $X$ and $a,v\in Y$, $v\neq0$. Let us consider the line $L = a + \er v$. Then it follows immediately from the definition above that $L\subset Y$ and that $f$ is essentially smooth on the line $L$ \ifff $f\upharpoonright_Y$ is.
\end{remark}

The theorem proven in [\ref{zajicDif2}; Proposition 3.3] is as follows.

\begin{thm}\label{tZajic}
 Let $X = X_1 \oplus \ldots \oplus X_n$ be a Banach space with a separable dual $X^*$. Let $G\subset X$ be an open set and $f:G\to\er$ a locally Lipschitz function. Let, for each $1\leq i\leq n$, there exists a dense set $D_i\subset S_{X_i}$ such that, for each $v\in D_i$, $f$ is essentially smooth on a generic line parallel to $v$. Then $f$ is generically Fr\'echet differentiable on $G$.
\end{thm}

Using the concept of rich families, it is proven in [\ref{zajicDif2}; Theorem 5.2] that this result holds with a slightly stronger assumptions even in the case of nonseparable Asplund spaces. Using the technic of elementary submodels we will prove, that the theorem \ref{tZajic} holds in exactly the same form in nonseparable Asplund spaces.

For the purpose of proving such a theorem, let us first begin with one lemma.

\begin{lemma}\label{lProj}\pp
 Let $X$ be a NLS, $X = X_1 \oplus \ldots \oplus X_n$. Let $P_1,\ldots,P_n$ be the corresponding projections onto subspaces $X_1,\ldots,X_n$. \ppXJ[$P_1,\ldots,P_n$]
  $$X_M = P_1(X_M)\oplus\ldots\oplus P_n(X_M).$$
\end{lemma}
\begin{proof}
 \letsMX, $P_1,\ldots,P_n$. Then according to the proposition \ref{pRngInM} it is true that $P_i(X\cap M)\subset X\cap M$ for each $i\in\{1,\ldots,n\}$. From the continuity of projections $P_1,\ldots,P_n$ it follows that $P_i(X_M)\subset X_M$ for each $i\in\{1,\ldots,n\}$.
 Consequently, $X_M = P_1(X_M)\oplus\ldots\oplus P_n(X_M)$.
\end{proof}

\begin{thm}
 Let $X = X_1 \oplus \ldots \oplus X_n$ be an Asplund space. Let $G\subset X$ be an open set and $f:G\to\er$ a locally Lipschitz function. Let, for each $1\leq i\leq n$, there exists a dense set $D_i\subset S_{X_i}$ such that, for each $v\in D_i$, $f$ is essentially smooth on a generic line parallel to $v$. Then $f$ is generically Fr\'echet differentiable on $G$.
\end{thm}
\begin{proof}
 Let $P_1,\ldots,P_n$ be the continuous projections onto subspaces $X_1,\ldots,X_n$. According to the corrolary \ref{cDif}, propositions \ref{pDense}, \ref{pCountableSubset}, \ref{pIsSubspace} and lemma \ref{lProj} we know, that there exists a list of formulas $\varphi_1,\ldots,\varphi_l$ and a countable set $Y$ such that for the set
 $$Z:=\{X,f,P_1,\ldots,P_n,D_1,\ldots,D_n,S_{X_1},\ldots,S_{X_n}, Y\}$$
 and for every elementary submodel $M$, $M\prec(\varphi_1,\ldots,\varphi_l;\;Z)$ it is true that:
\begin{itemize}
	\item[(P1)] Every countable set $S\in M$ is a subset of $M$.
	\item[(P2)] $X_M = P_1(X_M)\oplus\ldots\oplus P_n(X_M)$.
	\item[(P3)] Whenever sets $A, S\subset X$ are in $M$, the following holds:
	$$A\cap S\text{ is dense in }S \iff A\cap S\cap X_M\text{ is dense in }S\cap X_M.$$
	\item[(P4)] $D(f)$ is residual in $G \leftrightarrow D(f\upharpoonright_{X_M})$ is residual in $G\cap X_M$.
	\item[(P5)] $X_M$ is separable subspace of $X$
\end{itemize}

Without loss of generality we may assume that the list of formulas $\varphi_1,\ldots,\varphi_l$ is subformula closed. Notice, that for every subspace $N$ of $X$ satisfying $N = P_1(N)\oplus\ldots\oplus P_n(N)$ it is true that $S_{X_i}\cap N = S_{P_i(N)}$. Really, this equality follows from the fact that $$S_{X_i}\cap N = S_X\cap X_i\cap N = S_X \cap X_i \cap P_i(N) = S_X \cap P_i(N) = S_{P_i(N)}.$$
Let us define inductively a sequence of elementary submodels $\{M_k\}_{k\in\omega}$:
\begin{itemize}
	\item For $k = 0$ choose an arbitrary elementary submodel $M_0\prec(\varphi_1,...,\varphi_n;\; Z)$.
	\item Whenever $M_k$ is defined, we will find for every $i\in\{1,\ldots,n\}$ countable subset $C_{k,i}$ of $D_i\cap X_{M_k}$ dense in $S_{P_i(X_{M_k})} = S_{X_i}\cap X_{M_k}$. Then for every $v\in C_{k,i}$ it follows from the assumptions and lemma \ref{lLines} that the set $\{a\in G:\; f\text{ is }esm\text{ on the line }a + \er v\}$ is residual. Consequently, there exists a $G_\delta$ dense subset $G_{k,v}$ such that $f$ is $esm$ on each line parallel to $v$, intersecting $G_{k,v}$.\\
	Now we let $M_{k+1}$ to be an elementary submodel for formulas $\varphi_1,\ldots,\varphi_l$ containing $\{Z, C_{k,1},\ldots,C_{k,n},M_k,\{G_{k,v}\}_{v\in\bigcup_{i=1}^n C_{k,i}}\}$.
\end{itemize}
Finally, we define $M:=\bigcup_{k\in\omega}M_k$. Then according to the lemma \ref{lCupM}, $M\prec(\varphi_1,...,\varphi_n;\; Z)$. Therefore, $(P1) - (P5)$ holds for $M$.

We need to verify, that for the space $X_M$ and function $f\upharpoonright_{X_M}$ the conditions of the theorem \ref{tZajic} are satisfied. Then according to $(P4)$ it is true that $f$ is generically Fr\'echet differentiable on $G$.

Since $X$ is an Asplund space, $(X_M)^*$ is separable. Obviously, $f\upharpoonright_{X_M}$ is locally Lipschitz. According to $(P2)$, $X_M = P_1(X_M)\oplus\ldots\oplus P_n(X_M)$. For $i\in \{1,\ldots,n\}$ we define $C_i:=\bigcup_{k\in\omega}C_{k,i}$. Let us verify, that this set is dense in $S_{P_i(X_M)} = S_{X_i}\cap X_M$.

Fix an arbitrary $\eps > 0$ and $y\in S_{X_i}\cap X_M = S_{X_i} \cap \ov{\bigcup_{k\in\omega}(X\cap M_k)}$. Then find some $y_0 \in U(y,\tfrac{\eps}{3})\cap \bigcup_{k\in\omega}(X\cap M_k)$ and take $k\in\omega$ such that $y_0\in X\cap M_k$. Then $\frac{y_0}{\|y_0\|} \in X_{M_k}\cap S_{X_i}$. Furthermore,
\begin{align*}
\|\tfrac{y_0}{\|y_0\|} - y\| & \leq \|\tfrac{y_0}{\|y_0\|} - y_0\| + \|y_0 - y\| = |1 - \|y_0\|| + \|y_0 - y\| \\
														& = |\|y\| - \|y_0\|| + \|y_0 - y\| \leq 2\|y_0 - y\| < \tfrac{2\eps}{3}.
\end{align*}
Because $C_{k,i}$ is dense in $S_{X_i}\cap X_{M_k}$, there exists $c_{k,i}\in C_{k,i}\subset C_i$ such that $\|c_{k,i} - \frac{y_0}{\|y_0\|}\| < \frac{\eps}{3}$. Consequently,
$$\|c_{k,i} - y\|\leq \|c_{k,i} - \tfrac{y_0}{\|y_0\|}\| + \|\tfrac{y_0}{\|y_0\|} - y\| < \eps.$$

Notice that thanks to $(P1)$, $C_i\subset M$ for every $i\in\{1,\ldots,n\}$. It remains to show that for every $i\in\{1,\ldots,n\}$ and $v\in C_i$ the set
$$
R_{v}:=\{a\in G\cap X_M:\; f\upharpoonright_{X_M}\text{ is }esm\text{ on the line }a + \er v\}
$$
is residual in $X_M$.

Fix an arbitrary $v\in C_i$ and find $k\in\omega$ such that $v\in C_{k,i}$. Then $R_v\supset G_{k,v}\cap X_M$. Because $G_{k,v}\in M$, we get from $(P3)$ that $G_{k,v}\cap X_M$ is dense $G_\delta$ set in $X_M$. Consequently, $R_v$ is residual in $X_M$.
\end{proof}

The second application extends validity of the result from [\ref{preis}; Theorem 4.8] marked in this text as theorem \ref{tPreis}. This theorem was proven for spaces $\C(K)$ where $K$ is a countable compact and for subspaces of $c_0$. We will show how to use the technic of elementary submodels to get the same result for spaces $\C(K)$ where $K$ is a scattered compact and for subspaces of $c_0(\Gamma)$ for possibly uncountable set $\Gamma$.

Recall, that a set $A\subset T$ (where $T$ is an arbitrary topological space) is called {\em scattered}, if every nonempty subset has an isolated point. We wil need the following well-known fact.

\begin{lemma}\label{lScattered}
 Let $K$, $L$ be compact spaces, $K$ scattered, $L$ metrizable and $f:K\to L$ continuous mapping onto $L$. Then $L$ is a countable set.
\end{lemma}

Recall that a Banach space $Y$ is said to {\em have the Radon-Nikod\'ym property} (RNP) if every Lipschitz function $f:\er\to Y$ is differentiable almost everywhere (or equivalently every such $f$ has a point of differentiability - see \cite{preis}).

The result of J.Lindenstrauss and D.Preiss uses the notion of $\Gamma$-null sets. Therefore, let us give some basic notations. For further information about this notion see [\ref{tiser}, chapter 5].

Let $X$ be a Banach space and let $T:=[0,1]^\en$ be endowed with the product topology and product Lebesgue measure $\L^\en$. We denote by $\Gamma(X)$ the space of continuous mappings
$$\gamma:T\to X$$
having continuous partial derivatives $D_j\gamma$ (we consider one-sided derivatives at points where $j$-th coordinate is 0 or 1). We equip $\Gamma(X)$ with the topology generated by the seminorms
$$\|\gamma\|_\infty = \sup_{t\in T}\|\gamma(t)\|\quad\text{and}\quad\|\gamma\|_k = \sup_{t\in T}\|D_k\gamma(t)\|,k\geq 1.$$
Equivalently, this topology may be defined by the seminorms
$$\|\gamma\|_{\leq k} = \max\{\|\gamma\|_\infty,\|\gamma\|_1,\ldots,\|\gamma\|_k\}.$$
The space $\Gamma(X)$ with this topology is a Fr\'echet space; in particular it is a Polish space whenever $X$ is separable.

We define also $\Gamma_n(X) = \C^1([0,1]^n,X)$ and consider the norm $\|\cdot\|_{\leq n}$ on this space. Notice, that $\Gamma_n(X)$ is a subspace of $\Gamma(X)$ in the sense that the functions depending on the first $n$ coordinates only are naturally identified with the functions from $\Gamma_n(X)$.

A Borel subset $A\subset X$ is called $\Gamma$-null if the set $\{\gamma\in\Gamma(X);\;\L^\en\gamma^{-1}(A)=0\}$ is residual in $\Gamma(X)$.

It comes from [\ref{tiser}; Lemma 5.3.2 and Lemma 5.4.1] that the following two lemmas hold.
\begin{lemma}\label{lGammaDense}
 Whenever ($X_n$) is an increasing sequence of subspaces of $X$ whose union is dense in $X$, then $\bigcup_{n=1}^\infty\Gamma_n(X_n)$ is dense in $\Gamma(X)$.
\end{lemma}

\begin{lemma}\label{lGammaBorel}
 Let $A$ be a Borel subset of a Banach space $X$.\\Then the set $\{\gamma\in\Gamma(X);\;\L^\en\gamma^{-1}(A)=0\}$ is Borel.
\end{lemma}

The result from [\ref{preis}; Theorem 4.8] comes as follows.

\begin{thm}\label{tPreis}
 The following spaces have the property that every Lipchitz mapping of them into space with the RNP is Fr\'echet differentiable everywhere except a $\Gamma$-null set: $\C(K)$ for countable compact $K$, subspaces of $c_0$.
\end{thm}

Let us first focus on the set property ``to be $\Gamma$-null''. For those purposes we give the following lemmas.
\begin{lemma}
 Let $X$ be a finite dimensional Banach space and let $\{x_1,\ldots,x_n\}$ be basis of $X$. Then for every $k\in\omega$, $$\Gamma_k(X) = \{\textstyle{\sum_{i=1}^n} \gamma_i x_i;\;\gamma_i\in\Gamma_k(\er)\}.$$
\end{lemma}
\begin{proof}
 For every $k\in\omega$, $\gamma\in\Gamma_k(X)$ and $t\in[0,1]^k$ there are unique numbers $\gamma_1(t),\ldots,\gamma_n(t)$ such that $\gamma(t) = \sum_{i=1}^n \gamma_i(t)x_i$. It is easy to verify that for every $i\in\{1,\ldots,n\}$ the mapping $\gamma_i$ is an element of $\Gamma_k(\er)$ and that $D_j\gamma(t) = \sum_{i=1}^n D_j\gamma_i(t)x_i$ whenever $j\in\{1,\ldots,k\}$ and $t\in[0,1]^k$. Thus,
 $\Gamma_k(X) = \{\textstyle{\sum_{i=1}^n} \gamma_i x_i;\;\gamma_i\in\Gamma_k(\er)\}$.
\end{proof}

\begin{lemma}
 Let $X$ be a separable Banach space with a countable dense set $D$. Then $$\Gamma(X) = \ov{\{\textstyle{\sum_{i=1}^n} \gamma_i x_i;\;\gamma_i\in\Gamma_n(\er), x_i\in D, n\in \en\}}.$$
\end{lemma}
\begin{proof}
 Let us denote by $N$ either the dimension of $X$ if it is finite, or $N = \en$ if $X$ is infinite dimensional. Then take a countable linearly dense set $\{x_n\}_{n\in N}\subset D$ which is linearly independent. Denote by $X_n$ the subspace $\sspan\{x_i;i\leq n\}$. Then according to the preceeding lemma and lemma \ref{lGammaDense}, the set $\{\textstyle{\sum_{i=1}^n} \gamma_i x_i;\;\gamma_i\in\Gamma_n(\er), n\in \en\}$ is dense in $\Gamma(X)$.
\end{proof}

\begin{remark}
The preceeding lemma holds even in the case when $X$ is non-separable (with uncountable set $D:=X$). This is because the range of every $\gamma\in\Gamma(X)$ is separable. Thus, considering that $\gamma\in\Gamma(\ov{\sspan}\{\rng(\gamma)\})$, we may use the result for separable spaces.
\end{remark}

\begin{lemma}\label{lGamma}\pp
 Let $X$ be a Banach space. Then whenever $M$ contains $X$ and $\{\Gamma_n(X)\}_{n=1}^\infty$, it is true that
 $$\ov{\Gamma(X)\cap M} = \Gamma(X_M)$$
\end{lemma}
\begin{proof}
 \letsMX, $\{\Gamma_n(\er)\}^\infty_{n=1}$ and $\{\Gamma_n(X)\}^\infty_{n=1}$ (it is not necessary to mention the set $\{\Gamma_n(\er)\}^\infty_{n=1}$ in the assumptions of the lemma as it does not depend on the space $X$ - see Convention on the page \pageref{conventionM}). Then, according to the proposition \ref{pRngInM}, $\Gamma(X)\cap M\subset\Gamma(X_M)$; consequently, $\ov{\Gamma(X)\cap M}\subset\Gamma(X_M)$.

 For the other inclusion, denote for every $n\in\en$ $$A_n:= \{\textstyle{\sum_{i=1}^n} \gamma_i x_i;\;\gamma_i\in\Gamma_n(\er), x_i\in X\cap M\}.$$ Using the preceeding lemma, it is sufficient to show that for every $n\in\en$, $A_n\subset \ov{\Gamma(X)\cap M}$.
 Let us fix $n\in\en$. Using the absoluteness of the formula (for every $n\in\en$ the formula is the same - what does change is the free variable $\Gamma_n(\er)$ in it)
  $$(*)\qquad(\exists D)(D\text{ is countable and dense in }\Gamma_n(\er)),$$
 we may find a countable set $D\in M$ such that $D$ is dense in $\Gamma_n(\er)$. Besides that, whenever $\gamma_0\in \Gamma(R)\cap M$ and $x_0\in X\cap M$, than $\gamma_0 x_0$ is a function uniquely defined by the formula
  $$(*)\qquad(\exists f\in \Gamma_n(X))(\forall t\in[0,1]^n)(f(t) = \gamma_0(t) x_0);$$
 consequently, $\gamma_0 x_0\in M$. As the space $\Gamma(X)\cap M$ is $\qe$-linear, it is true that $\{\sum_{i=1}^n \gamma_i x_i;\;\gamma_i\in D, x_i\in X\cap M\}\subset \ov{\Gamma(X)\cap M}$. It is easy to verify that this subset of $\ov{\Gamma(X)\cap M}$ is dense in $A_n$.
\end{proof}

\begin{remark}
 The preceeding lemma is interesting by itself. Observe, that combining it with the results from previous sections we get, that for every suitable elementary submodel and for every set $A\subset \Gamma(X)$ contained in $M$ it is true that $A$ is dense (resp. nowhere dense) in $\Gamma(X)$ \ifff $A\cap \Gamma(X_M)$ is dense (resp. nowhere meager) in $\Gamma(X_M)$. When $A$ has the Baire property, then the same equivalence holds for the residuality of $A$. This result gives us separable subspaces with properties that were not achieved in \cite{tiser}  using the technic of rich families (see [\ref{tiser}; Lemma 5.6.1]).
\end{remark}

\begin{cor}\label{cGammaNull}\pp
 Let $X$ be a Banach space. \ppXD[$\{\Gamma_n(\er)\}_{n=1}^\infty$]{a Borel set $A$}
 $$A\text{ is }\Gamma\text{-null in }X \iff A\cap X_M\text{ is }\Gamma\text{-null in }X_M.$$
\end{cor}
\begin{proof}
 \letsMX, $\{\Gamma_n(\er)\}_{n=1}^\infty$\} and a Borel set $A$. Then, using the lemma \ref{lGammaBorel} and \ref{lGamma}, $\{\gamma\in\Gamma(X);\;\L^\en\gamma^{-1}(A)=0\}$ is residual in $\Gamma(X)$ \ifff $\{\gamma\in\Gamma(X_M);\;\L^\en\gamma^{-1}(A\cap X_M)=0\}$ is residual in $\Gamma(X_M)$.
\end{proof}

Using the preceeding results, we can put forward the promised extension of the theorem \ref{tPreis}.

\begin{thm}
 The following spaces have the property that every Lipchitz function of them into space with the RNP is Fr\'echet differentiable everywhere except a $\Gamma$-null set: $\C(K)$ for scattered compact $K$, subspaces of $c_0(\Gamma)$, where $\Gamma$ is an arbitrary set.
\end{thm}
\begin{proof}
 Let us have a space $X$ from the assumptions (either $X = \C(K)$ for scattered compact $K$, or $X\subset c_0(\Gamma)$), a Banach space $Y$ with RNP and a Lipschitz function $f:X\to Y$. Using the preceeding corrolary \ref{cGammaNull} and the theorem \ref{tDif}, choose an elementary submodel $M$ satisfying:
 \begin{itemize}
  \item $X_M$ is a separable subspace of $X$
  \item $f$ is Fr\'echet differentiable everywhere except a $\Gamma$-null set in $X$ if and only if $f\upharpoonright_{X_M}$ is Fr\'echet differentiable everywhere except a $\Gamma$-null set in $X_M$
 \end{itemize}

 If $X = \C(K)$, then choose (using lemma \ref{lCKM}) such an elementary submodel $M$, that in addition it holds that $X_M = \C(K/_M)$, where $K/_M$ is metrizable compact and a continuous image of $K$. From the lemma \ref{lScattered} it follows that $K/_M$ is a countable compact. Then from the theorem \ref{tPreis} it follows that $f\upharpoonright_{X_M}$ is Fr\'echet differentiable everywhere except a $\Gamma$-null set in $X_M$, so $f$ is Fr\'echet differentiable everywhere except a $\Gamma$-null set.

 If $X = c_0(\Gamma)$, then $X_M$ is a separable subspace of $X$, so $X_M$ is a subspace of $c_0$. Then, using the same arguments as above, $f$ is Fr\'echet differentiable everywhere except a $\Gamma$-null set.
\end{proof}

\begin{ack}
 The author would like to thank Ond\v{r}ej Kalenda for suggesting the topic and for many useful remarks. My gratitude belongs also to J. Ti\v{s}er for useful remark which helped to prove the lemma \ref{lGamma}.
\end{ack}

\end{document}